\documentclass[a4paper,12pt]{amsart}
\usepackage{mathtools}
\makeatletter
\@addtoreset{footnote}{page}
\makeatother

\usepackage{amsthm}
\usepackage{amsmath,amssymb,latexsym,amsfonts,mathrsfs}
\usepackage[dvipdfmx]{color,graphicx}

\renewcommand{\Re}{\mathop{\rm Re}}

\newcommand{\sgn}{\mathop{\rm sgn}}

\renewcommand{\bar}{\overline}

\newcommand{\bC}{\ensuremath{\mathbb{C}}}
\newcommand{\bD}{\ensuremath{\mathbb{D}}}
\newcommand{\bE}{\ensuremath{\mathbb{E}}}

\newcommand{\bN}{\ensuremath{\mathbb{N}}}

\newcommand{\bP}{\ensuremath{\mathbb{P}}}

\newcommand{\bR}{\ensuremath{\mathbb{R}}}
\newcommand{\bS}{\ensuremath{\mathbb{S}}}

\newcommand{\bZ}{\ensuremath{\mathbb{Z}}}

\newcommand{\cA}{\ensuremath{\mathcal{A}}}
\newcommand{\cB}{\ensuremath{\mathcal{B}}}

\newcommand{\cF}{\ensuremath{\mathcal{F}}}

\theoremstyle{plain}
\newtheorem{Thm}{Theorem}[section]

\newtheorem{Lem}[Thm]{Lemma}
\newtheorem{Prop}[Thm]{Proposition}
\newtheorem{Cor}[Thm]{Corollary}

\theoremstyle{definition}
\newtheorem{Asmp}[Thm]{Assumption}

\newtheorem{Def}[Thm]{Definition}
\newtheorem{Rem}[Thm]{Remark}
\newtheorem{Ex}[Thm]{Example}

\numberwithin{equation}{section}

\makeatletter
\renewcommand\section{\@startsection {section}{1}{\z@}%
                                   {-3.5ex \@plus -1ex \@minus -.2ex}%
                                   {2.3ex \@plus.2ex}%
                                   {\normalfont\large\bf}}
\makeatother

\makeatletter
\renewcommand\subsection{\@startsection {subsection}{1}{\z@}%
                                   {-3.5ex \@plus -1ex \@minus -.2ex}%
                                   {2.3ex \@plus.2ex}%
                                   {\normalfont\normalsize\bf}}
\makeatother

\usepackage{bbm}
\usepackage{hyperref}
\mathtoolsset{showonlyrefs=true}

\newcommand{\rd}{\ensuremath{\mathrm{d}}}

\begin{document}

\title{Generalized uniform laws for 
tied-down occupation times of 
infinite ergodic transformations}
\date{\today}
\author{Jon. Aaronson \and Toru Sera}
\begin{abstract} We establish a conditional limit theorem for occupation times of infinite ergodic transformations under a tied-down condition, that is, the condition that the orbit returns to a reference set with finite measure at the final observation time. The class of limit distributions is the generalization of the uniform distribution which was discovered by M.\ Barlow, J.\ Pitman and M.\ Yor in [\textit{S\'eminaire de Probabilit\'es XXIII. Lecture Notes in Mathematics}, volume 1372 (1989), 294--314]. 
%Our result is similar to generalized uniform laws for one-dimensional diffusion bridges. 
For the proof we utilize operator renewal theory.
Our result can be applied to intermittent maps with two or more indifferent fixed points.  
\end{abstract}

\address[Jon Aaronson]{School of Math. Sciences, Tel Aviv University,
69978 Tel Aviv, Israel.}
\email{aaro@tau.ac.il}
\address[Toru Sera]{Department of Mathematics, Graduate School of Science, Osaka University,
Toyonaka, Osaka 560-0043, Japan.}
\email{sera@math.sci.osaka-u.ac.jp}

\subjclass{Primary 37A40; Secondary 37A50, 60F05}
\keywords{infinite ergodic theory, operator renewal theory, intermittent maps, Barlow--Pitman--Yor generalized uniform distributions}

\maketitle
%\tableofcontents

%%%%% text %%%%%

\section{Introduction}

Paul L\'evy established arcsine and uniform laws for occupation times of one-dimensional Brownian motion in \cite{Levy39}.
The arcsine law is now classical in probability theory, and has been extended to a variety of classes of stochastic processes, such as 
random walks \cite{ErdKac, Spa54, Spi56},
renewal processes \cite{Lam58, FKY}, diffusion processes \cite{BPY, Wat95, Ya.Y17}, L\'evy processes \cite{GetSha94}, infinite ergodic transformations \cite{Tha02, ThZw, Zwe07cpt, SerYan19, Ser20} and  random dynamical systems \cite{HaYa, NNTY}. The uniform law has also been extended to, e.g., random walk bridges \cite{ChuFel, Lip52, Spa53}, diffusion bridges \cite{BPY, Ya.Y06, JLP, Jam10}
and L\'evy or exchangeable bridges \cite{FitGet95, Kni96}.

In this paper we generalize the uniform law for occupation times to infinite ergodic transformations under a tied-down condition, i.e., the condition that the orbit returns to a reference set with finite measure at the final observation time. 
In previous studies \cite{AarSer22, AarSer23+} we studied occupation times on a set with finite measure under the tied-down condition. 
In this paper we study occupation times on sets with infinite measure under the tied-down condition.
Our abstract result can be applied to so-called intermittent maps, more specifically, non-uniformly expanding interval maps with two or more indifferent fixed points.
Our result is similar to those of \cite{BPY, Ya.Y06}, but our approach is different, using operator renewal theory   as in \cite{AaDe, Sar02, Gou11, MelTer, AarSer22, AarSer23+} to deduce certain multi-dimensional local limit theorems and local large deviation estimates, from which we obtain the desired generalized uniform laws. 
The theory of regular variation plays an important role.

Before going into details,
let us recall L\'evy's arcsine and uniform laws.
Let $(B_s)_{s\geq0}$ be a one-dimensional Brownian motion with $B_0=0$ which is defined on a probability space $(\Omega, \cF, \bP)$, and let $S_t=\int_0^t 1_{\{B_s>0\}}\rd s$ be the occupation time of $(B_s)_{s\geq0}$ on the positive side up to time $t$.
Then, for each $t>0$ and $x\in[0,1]$,
\begin{equation}
\label{intro:arcsine-law}
   \bP\bigg[\frac{S_t}{t}\leq x\bigg]
   =
   \int_0^x \frac{\rd s}{\pi\sqrt{s(1-s)}}=
    \frac{2}{\pi}\arcsin(\sqrt{x}).
\end{equation}
In other words, the occupation time ratio $S_t/t$ has the arcsine distribution on $[0,1]$.
The equality \eqref{intro:arcsine-law} is called L\'evy's arcsine law for occupation times.
In addition,
\begin{equation}
\label{intro:uniform-law}
	\bP\bigg[
	\frac{S_t}{t}\leq x
	\:\bigg|\: B_t=0\bigg]=\int_0^x \rd s=x,
\end{equation}
where $\bP[A|B]$ denotes the conditional probability of $A$ given $B$. In other words, given $B_t=0$, the occupation time ratio $S_t/t$  has the uniform distribution on $[0,1]$. 
The equality \eqref{intro:uniform-law} is called L\'evy's uniform law for occupation times.
%Roughly speaking, it is not likely that $S_t/t$  is close to $0$ or $1$ because of the condition $B_t=0$.

Let us illustrate Thaler's generalized arcsine laws for intermittent maps \cite{Tha02}.

\begin{Ex}[arcsine law for Boole's transformation]\label{Ex:Gen-Arc}
Following \cite{Tha02},
let us define an interval map $f:[0,1]\to [0,1]$ by
\begin{align}
  f(x)=
  \begin{cases}\displaystyle
  \frac{x(1-x)}{1-x-x^2}, &0\leq x \leq 1/2,
  \\[10pt]
  \displaystyle
  1-f(1-x), & 1/2<x\leq 1.                                                                                                                                                                                                                                                                                                                                                                                                             
  \end{cases}
\end{align}
Then $f(0)=0$, $f(1)=1$, $f'(0)=f'(1)=1$ and 
\begin{equation}
f(x)-x=1-x-f(1-x)\sim x^3, 
\quad\text{as $x\to 0$}.	
\end{equation}
Therefore $0$ and $1$ are indifferent fixed points of $f$. The restrictions $f|_{(0,1/2)}:(0,1/2)\to(0,1)$ and $f|(1/2,1):(1/2,1)\to(0,1)$ are $C^2$-bijections and non-uniformly expanding.  
The map $f$ is conjugated to Boole's transformation $Tx=x-x^{-1}$ ($x\in\bR\setminus \{0\}$). 
Indeed, let $\phi(x)=(1-x)^{-1}-x^{-1}$ ($x\in(0,1)$), then $f=\phi^{-1}\circ T\circ \phi$ on $(0,1)\setminus\{1/2\}$. We refer the reader to \cite{Boole, AdWe} for the details of Boole's transformation.
The orbit $(f^k(x))_{k=0}^\infty$ exhibits an intermittent behavior, i.e., it stays near $0$ or $1$ for long time, and it escapes from $0$ and $1$ intermittently but returns near $0$ or $1$ immediately. See \cite[Section 1]{Ser20} for the figure of the intermittent behavior. 
The map $f$ admits an ergodic invariant measure 
\begin{equation}
\mu(\rd x)
=
\bigg(\frac{1}{x^2}+\frac{1}{(1-x)^2}\bigg)\rd x,
\quad
0\leq x \leq 1.
\end{equation} 
Take a $2$-periodic point $\gamma=\sqrt{2}-1\in(0,1/2)$. Then $f(\gamma)=2-\sqrt{2}\in(1/2,1)$ and $f^2(\gamma)=\gamma$. Set 
\begin{equation}
A_1=[0,\gamma),
\quad 
Y=[\gamma,f(\gamma)]
\quad \text{and} \quad 
A_2=(f(\gamma), 1].
\end{equation}
Then $\mu(A_j)=\infty$ and $\mu(Y)\in(0,\infty)$.
For $A\subset[0,1]$, denote by 
\begin{align}
	S_{A}(n)=\sum_{k=1}^n 1_{A}\circ f^k
\end{align}
the occupation time on $A$ of the orbit $(f^k (x))_{k=0}^\infty$ from time $1$ to time $n$. Then 
\begin{align}
\frac{S_{A_1}(n)+S_{A_2}(n)}{n}
=
1-\frac{S_Y(n)}{n} \to 1, \quad\text{a.e., as $n\to\infty$},	
\end{align}
 which follows from Birkhoff's ergodic theorem. 
We are interested in the asymptotic behavior of $S_{A_1
}(n)/n$. 
We now explain a typical example of Thaler's generalized arcsine laws \cite{Tha02}.
Let $\nu(\rd x)$ be a probability measure on $[0,1]$ absolutely continuous with respect to $\mu(\rd x)$ (or equivalently, with respect to the Lebesgue measure $\rd x$). We interpret $\nu(\rd x)$ as the distribution of the initial point $x$ of the orbit $(f^k (x))_{k=0}^\infty$.  Then
\begin{align}\label{intro:gen-arc}
    \lim_{n\to\infty}\nu\bigg[\frac{S_{A_1}(n)}{n}\leq t \bigg]
    =
    \int_0^t \frac{\rd s}{\pi\sqrt{s(1-s)}}=
    \frac{2}{\pi}\arcsin(\sqrt{t}),
    \quad t\in[0,1].	
\end{align}

%For the proof, it is important to make use of the first return time
%\begin{align}
%\varphi(x)=\min\{k\geq1: f^k(x)\in Y\},\quad
%    x\in [0,1],
%\end{align}
%similarly as in the renewal theory.
%It is worth noting that there exists a constant $C>0$ such that, for $j=1,2$,
%\begin{align}
%\mu(Y\cap \{S_{A_j}(\varphi)=n\})\sim C n^{-3/2}, \quad \text{as $n\to\infty$,}
%\end{align}
%and hence
%\begin{equation}
%	\mu(Y\cap \{S_{A_j}(\varphi)\geq n\})\sim 2C n^{-1/2}, \quad \text{as $n\to\infty$.}
%\end{equation}
%Indeed Thaler \cite{Tha02} utilized  this kind of regular variations for the proof of his generalized arcsine law.
\end{Ex}

We now illustrate our result. 

\begin{Ex}[uniform law for Boole's transformation]
Under the setting of Example \ref{Ex:Gen-Arc}, we further assume  $\nu(\rd x)$ admits a Riemann-integrable density function with respect to the Lebesgue measure on $[0,1]$. As a typical example of our main result, we obtain 
\begin{equation}\label{intro:unif}
	\lim_{n\to\infty}\nu\bigg[\frac{S_{A_1}(n)}{n}\leq t \:\bigg|\: f^{-n}(Y) 
	\bigg]
	=
	\int_0^t \rd s=t , 
	\quad\text{$t\in[0,1]$,}
\end{equation}
where $\nu[A|B]=\nu[A\cap B]/\nu[B]$ denotes the conditional probability of $A$ given $B$. We remark that $f^{-n}(Y)$ is a rare event. More specifically, there exists a constant $C>0$ such that
\begin{equation}
  \nu[f^{-n}(Y)]\sim C n^{-1/2},
  \quad
  \text{as $n\to\infty$.}	
\end{equation}
See Section \ref{sec:intermittent} for the details. 
%Hence it is natural that conditional probabilities given $f^{-n}(Y)$ have different properties from unconditional probabilities. 
\end{Ex}

This paper is organized as follows.
In Section \ref{sec:limit}, we recall a generalization of the uniform distribution.
In Section \ref{sec:setting}, we set up notations and terminology of infinite ergodic theory and operator renewal theory. 
In Section \ref{sec:main}, we state our main results in abstract setting.
In Section \ref{sec:asymp}, we state and prove a multi-dimensional  local limit theorem and a multi-dimensional local large deviation estimate by using operator renewal theory. 
In Sections \ref{sec:proof} and \ref{sec:proof2}, we use the local limit theorem and the local large deviation estimate to prove our main results.
In Section \ref{sec:intermittent}, we apply our abstract results to intermittent maps.

\section{Barlow--Pitman--Yor generalized uniform distributions}\label{sec:limit}

In this section, we use independent one-sided stable random variables to introduce the generalizations of the uniform distribution  
which appeared in \cite[Th\'eor\`eme 2]{BPY} as joint distributions of occupation times and local times of skew Bessel diffusion bridges on multi-ray.

First of all, we recall a relationship between one-sided stable distributions and Mittag-Leffler distributions.
Let $(\Omega,\cF,\bP)$ be a probability space. 
Let $\alpha\in(0,1)$ and $c\in(0,\infty)$ be constants. Let $\xi$ be a one-sided $\alpha$-stable random variable with Laplace transform
\begin{equation}
\bE[e^{-\lambda\xi}]=\exp(-\lambda^\alpha c), \quad\lambda\geq0.
\end{equation}
Its characteristic function is given by
\begin{equation}
\bE[e^{is\xi}] 
	= 
	\exp\bigg[\!
	-|s|^\alpha c
	\bigg(\cos\bigg(\frac{\pi\alpha}{2}\bigg)-i\sgn(s)\sin\bigg(\frac{\pi\alpha}{2}\bigg)\bigg)
	\bigg],
	\quad s\in\bR,
\end{equation}
where $i=\sqrt{-1}$.
By \cite[Theorem 11]{Fel49},
 $\xi^{-\alpha}$ has an $\alpha$-Mittag--Leffler distribution with
\begin{equation}
	\bE[\exp(z \xi^{-\alpha})]
	=
	\sum_{n\geq0}\frac{(z/c)^n}{\Gamma(1+n\alpha)}, \quad z\in\bC. 
\end{equation}
Note that the $n$th moment of $\xi^{-\alpha}$ is given by $\bE[\xi^{-n\alpha}]=n!/(c^n \Gamma(1+n\alpha))$.
See also \cite[p.453]{Fel71} and \cite[Example 24.12]{Sat} for the details.

%Denote by $\bE[\cdot]$ the expectation with respect to $\bP$.
Let $d\in\bN$, $\alpha\in(0,1)$ and $\beta=(\beta_1,\dots,\beta_d)\in(0,1)^d$ with $\sum_{j=1}^d \beta_j=1$.
We denote by $\xi_1,\dots,\xi_d$ independent one-sided $\alpha$-stable random variables with Laplace transforms 
\begin{align*}
\bE[\exp(-\lambda\xi_j)]=\exp(-\lambda^\alpha \beta_j), \quad\lambda\geq0, \;j=1,\dots,d.
\end{align*}
Its characteristic function $\bE[e^{is\xi_j}]$ is given by
\begin{align}\label{eq:Fourier-xi}
	\bE[e^{is\xi_j}] 
	= 
	\exp\bigg[\!
	-|s|^\alpha\beta_j
	\bigg(\cos\bigg(\frac{\pi\alpha}{2}\bigg)-i\sgn(s)\sin\bigg(\frac{\pi\alpha}{2}\bigg)\bigg)
	\bigg],
	\quad s\in\bR.
\end{align}
The random variable $\xi_j$ admits a continuous version of probability density function, which will be denoted by $\psi_j(y)$, $y\geq0$. The Fourier inversion formula implies that
\begin{align}\label{def:psi_j}
    \psi_j(y) 
    =
    \frac{\bP[\xi_j\in \rd y]}{\rd y}
    = 
    \frac{1}{2\pi}\int_{\bR} \bE[e^{is\xi_j}]e^{-isy} \rd s,
    \quad y\geq0.
\end{align}
Set 
\begin{align}
\eta=\xi_1+\dots+\xi_d,
\end{align}
which is a one-sided $\alpha$-stable random variable with Laplace transform $
	\bE[\exp(-\lambda \eta)]=\exp(-\lambda^\alpha)$, $\lambda \geq0$. 
%where $\Gamma(p)=\int_0^\infty e^{-t}t^{p-1}\rd t$ ($p>0$) denotes the gamma function. 
Note that $\bE[\eta^{-\alpha}]=1/\Gamma(1+\alpha)$.
%Our normalization of the Mittag-Leffler distribution is the same as in \cite{DK, Wen64, BPY, Sat}, but differs from \cite{Aa97}.

We now introduce Barlow--Pitman--Yor generalized uniform distributions.

\begin{Def}\label{DEF:gen-unif-dist}
 Denote by $(U_1,\dots,U_d,W)$ a $[0,1]^d \times[0,\infty)$-valued random variable whose joint distribution is characterized by the following: for any bounded measurable function $g:[0,1]^d\times [0,\infty)\to \bR$, 
 \begin{align}\label{def:gen-unif}
 \bE\big[g(U_1,\dots,U_d, W\big)\big]
	=
	\bE\bigg[g\bigg(\frac{\xi_1}{\eta}, \dots, \frac{\xi_d}{\eta}, \frac{1}{\eta^\alpha} \bigg) \frac{\Gamma(1+\alpha)}{\eta^\alpha}\bigg].
\end{align}
We call the joint distribution of $(U_1,\dots,U_d,W)$ the $(\alpha,\beta_1,\dots,\beta_d)$-Barlow--Pitman--Yor generalized uniform distribution.
\end{Def}

The right-hand side of \eqref{def:gen-unif} can be rewritten as
\begin{align}
	\int_{[0,\infty)^d}	
    g\bigg(\frac{y_1}{\sum_j y_j},\dots,\frac{y_d}{\sum_j y_j},\frac{1}{(\sum_j y_j)^\alpha}\bigg)
	\frac{\Gamma(1+\alpha)}{(\sum_j y_j)^\alpha}
	 \bigg(\prod_j \psi_j(y_j)\bigg)\rd y_1\cdots\rd y_d.	
\end{align}
%Note that the right-hand side of \eqref{def:gen-unif} is a positive linear functional with respect to $g$.
%In the case of $g\equiv 1$, it is equal to $\Gamma(1+\alpha)\bE[\eta^{-\alpha}]=1$. Hence Definition \ref{DEF:gen-unif-dist} is well-defined by virtue of the Riesz--Markov--Kakutani representation theorem.
It is easily seen that $\sum_j U_j=1$, a.s. 
In the case  $\alpha=1/2$, the law of $U_1$ is given by
\begin{equation}
	\bP[U_1\in \rd x]
	=
	\frac{\beta_1(1-\beta_2)}{2}
	\Big(
	\beta_1^2(1-x)+(1-\beta_1)^2
	x
	\Big)^{-3/2}
	\rd x,
	\quad 0\leq x\leq 1,
\end{equation}
as stated in \cite[Th\'eor\`eme 3]{BPY}.
See also \cite[Example 1]{Ya.Y06}.
In the special case  $\alpha=\beta_1=1/2$, the  law of $U_1$ is the usual uniform distribution on $[0,1]$. For general parameters $(\alpha,\beta)$, the law of $U_1$ is also characterized by the Stieltjes transform 
\begin{equation}
	\bE[(\lambda + U_1)^{-\alpha}]= (\beta_1(1+\lambda)^\alpha +(1-\beta_1)\lambda^\alpha)^{-1},
	\quad
	\lambda>0.
\end{equation}
See \cite{Ya.Y06, YaYa08} for details of $U_1$, such as its distribution function $G_{\alpha,\beta_1}$ and its density function $G'_{\alpha, \beta_1}$. The law of $W$ also appeared in \cite[Theorem 6]{Wen64}.

\begin{Rem}
The joint-distribution and marginal distributions of
\begin{equation}
    \Big(\frac{\xi_1}{\eta},\dots,\frac{\xi_d}{\eta},\frac{1}{\eta^\alpha}\Big)	
\end{equation}
have appeared as the limit distributions of occupation times of Markov chains \cite{DK, Lam58},
  diffusion processes \cite{Kas, BPY, Wat95, Ya.Y17, JLP, Jam10} and infinite ergodic transformations \cite{Aa81, Aa86, Tha00, ThZw, Zwe07cpt, SerYan19, Ser20}.
The joint-distribution  of $(\xi_1/\eta,\dots,\xi_d/\eta)$ is a multi-dimensional version of Lamperti's generalized arcsine distribution. As mentioned above,  $\eta^{-\alpha}$ has an $\alpha$-Mittag-Leffler distribution. 
\end{Rem}

\section{Setting}\label{sec:setting}
In this section, we will prepare some notations and formulate our assumptions for the main result.
We emphasize that suitable intermittent interval maps satisfy our assumptions, as we shall see in Section \ref{sec:intermittent}.

Let $(X,\cA,\mu)$ be a $\sigma$-finite measure space with $m[X]=\infty$. 
Let $f:(X,\cA,\mu) \to (X,\cA,\mu)$ be a conservative, ergodic, measure preserving transformation (CEMPT for short). 
%Equivalently, we assume that $\mu\circ f^{-1}=\mu$ and
%\begin{align*}
%	\sum_{n\geq0}1_{\{f^n x\in A\}}=\infty,
%	\quad
%	\text{$\mu$-a.e.$x$., for any $A\in\cA$ with $\mu(A)>0$,}
%\end{align*}
%where $1_{\{\cdot\}}$ denotes the indicator function.
%Let us denote by $R:L^1(\mu)\to L^1(\mu)$ the transfer operator of $f$ with respect to $\mu$. Then we have $\int_X (Ru)v \rd m =\int_X u(v\circ f) \rd m$ for any $u\in L^{1}(m)$ and $v\in L^\infty(m)$. 
For the basic properties of CEMPT,  see \cite[Chapter 1]{Aa97}. Following \cite{ATZ, ThZw, Zwe07cpt, SerYan19, Ser20}, we impose the following assumption:

\begin{Asmp}[dynamical separation]\label{asmp:dyn-sep}
	Let $d\geq2$ be a positive integer. Assume that the following two conditions.
\begin{enumerate}
\item The state space $X$ can be decomposed as $X=Y\cup A_1 \cup \dots\cup A_d$, where $Y,A_1,\dots,A_d\in\cA$ are disjoint sets with $\mu[Y]=1$ and $\mu[A_1]=\dots=\mu[A_d]=\infty$.
\item $f^{-1}(A_i)\subset A_i\cup Y$ for $i=1,\dots,d$. In other words, the condition [$x\in A_i$ and $f^n (x) \in A_j$ for some $i\neq j$ and $n\geq 1$] implies the existence of $k<n$ for which $f^k (x)\in Y$.
\end{enumerate}  
\end{Asmp}
%Let us denote by $\mu_Y(\cdot):=m(\cdot\cap Y)$ the restriction of $\mu$ over $Y$. 
%From now on we always impose 

For $x\in X$, we denote the first return time to $Y$ by 
\begin{align}
\varphi=\varphi(x)=\min\{k\geq1: f^k (x) \in Y \}
<\infty,\quad\text{a.e.}
\end{align}
For a measurable set $A\in\cA$, we denote by 
\begin{align}
S_A(n)=\sum_{k=1}^{n}1_A\circ f^k
\end{align}
the occupation time on $A$ of the orbit $(f^k(x))_{k=0}^\infty$ from time $1$ to time $n$.
Because of Assumption \ref{asmp:dyn-sep}, we have
\begin{align}\label{eq:varphi_j}
 S_{A_j}(\varphi)
 =
	\begin{cases}
		\varphi-1, & \text{on $Y\cap f^{-1}(A_j)$,}
		\\
		0, &\text{on $Y\cap f^{-1}(A_j^c)$.}
	\end{cases}
\end{align}
For $j=1,\dots,d$ and $n\geq1$, set
\begin{align}
    r_{0,0}&
    = \mu[Y\cap\{\varphi=1\}]=\mu[Y\cap f^{-1}(Y)],
    \\
    r_{j,n}
    &=\mu[Y\cap \{S_{A_j}(\varphi)=n\}]=\mu\bigg[Y\cap \bigcap_{j=1}^n f^{-j}(A_j)\cap f^{-(n+1)}(Y)\bigg].	
\end{align}
It is easily seen that $\sum_{j, n}r_{j,n}=1$, where the sum is taken over $(j,n)=(0,0)$ and $(j,n)\in\{1,\dots,d\}\times \bN$.
%Note that $\varphi=\varphi_1+\dots+\varphi_d+1$.

%\begin{align}
% S_A(n)=S_A(n)(x):=\sum_{k=1}^n \mathbbm{1}_A(T^k x),
% \quad n\geq0, \; x\in X.	
%\end{align}

%Let $g, h$ be positive-valued measurable functions on $(0,\infty)$.  If $g(t)/h(t)\to 1$ as $t\to\infty$, then we write $g(t)\sim h(t)$ as $t\to\infty$. If $g(\lambda t)\sim g(t)$ as $t\to \infty$ for any $\lambda>0$, then we say that $g$ is a slowly varying function. 

%The following assumption is essentially needed for the strong distributional convergence of the number of visits to the sets $Y, A_1,\dots,A_d$. 

\begin{Asmp}[regular variation]\label{asmp:reg-var}
	Let $\alpha\in (0,1)$ and $\beta=(\beta_1,\dots,\beta_d)\in(0,1)^d$  with $\sum_{j=1}^d\beta_j=1$. 
	Let $\ell:[0,\infty)\to(0,\infty)$ be a measurable function slowly varying at $\infty$. For each $j=1,\dots,d$, it is assumed that
\begin{align}\label{asmp:reg-var-1}
    \sum_{k\geq n} r_{j,k} \sim \beta_j n^{-\alpha} \ell(n)
    \quad 
    \text{and}
    \quad
    r_{j,n} = O(n^{-\alpha-1}\ell(n)),
    \quad\text{as $n\to\infty$.}
\end{align}

%In addition, if $\alpha\in (0, 1/2]$ then it is required that 
%\begin{align}\label{asmp:reg-var-2}
%  	\sum_{j=1}^n r_{j,n} \;(= \mu[\varphi=n+1]) \sim  \alpha n^{-\alpha-1} \ell(n),
%    \quad\text{as $n\to\infty$.}
%\end{align}

\end{Asmp}

Because of \eqref{eq:varphi_j}, the condition \eqref{asmp:reg-var-1} implies the asymptotic relations 
\begin{align}
\mu[Y\cap \{\varphi>n\}]\sim n^{-\alpha}\ell(n)\quad \text{and} \quad\mu[Y\cap \{\varphi=n\}]=O(n^{-\alpha-1}\ell(n)).
\end{align}

Define the first return map $F:Y\to Y$ by
\begin{align}
F(x)=f^{\varphi(x)}(x).
\end{align}
Then $F$ is a CEMPT on $(Y,\cA\cap Y, \mu|_Y)$ where $\cA\cap Y=\{A\cap Y: A\in \cA\}$ and $\mu|_Y$ denotes the restriction of $\mu$ over $Y$.

Let us denote by $R:L^1(Y)\to L^1(Y)$ the transfer operator of  $F$ with respect to $\mu|_Y$. The operator $R$ is characterized by the equation
\begin{equation}
\int_Y (Ru)v\:\rd\mu=\int_Y u(v\circ F)\:\rd \mu,
\quad
u\in L^1(Y),\;v\in L^\infty (Y).
\end{equation}
For $j=1,\dots,d$ and $n\geq1$, we define a bounded linear operators $R_{0,0}:L^1(Y)\to L^1(Y)$ and $R_{j, n}:L^1(Y)\to L^1(Y)$ by
\begin{align}
	R_{0,0} u=R\Big(1_{\textstyle\{\varphi=1\}}u\Big),
	\quad
	R_{j, n}u=
	R\Big(1_{\textstyle \{S_{A_j}(\varphi)=n\}} u\Big),
	\quad u\in L^1(Y). 
\end{align}
Set 
\begin{align}
	\bD^d=\{(z_1,\dots,z_d)\in\bC^d:|z_1|,\dots,|z_d|<1\},
	\\ 
	\bar{\bD}^d=\{(z_1,\dots,z_d)\in\bC^d:|z_1|,\dots,|z_d|\leq 1\}.
\end{align}
We write $\bD^1=\bD$ and $\bar{\bD}^1=\bar{\bD}$ for short.
Given $z=(z_1,\dots,z_d)\in \bar\bD^d$, we define a bounded linear operator $R(z):L^1(Y)\to L^1(Y)$ by
\begin{align}
    R(z)u=\bigg(R_{0,0}+\sum_{\substack{j=1,\dots, d \\ n\geq1}}z_j^n R_{j, n}\bigg) u,
    \quad u\in L^{1}(Y),	
\end{align}
Note that  $R(1,\dots,1)=R$.

For a Banach space $(\cB, \|\cdot\|_{\cB})$, we denote by $\mathscr{L}(\cB)$ the class of bounded linear operators $T:\cB\to \cB$, which is endowed with the operator norm 
$\|T\|_{\mathscr{L}(\cB)}=\sup\{\|Tu\|_{\cB}:u\in\cB,\;\|u\|_{\cB}\leq 1\}$, $T\in\mathscr{L}(\cB)$.
Imitating \cite{Sar02, Gou11, MelTer}, we impose the following assumption:

\begin{Asmp}[aperiodic renewal sequence]\label{asmp:ape-ren}
Let $\cB\subset L^\infty(Y)$ be a Banach space. Suppose $\cB$ contains constant functions and $\|u\|_{L^\infty(Y)} \leq \|u\|_{\cB}$ for $u\in\cB$. Moreover assume that the following three conditions.
\begin{enumerate}
    \item \label{asmp:ape-ren-bdd} The operator $R_{j,n}$ can be regarded as an element of $\mathscr{L}(\cB)$. In addition there exists $C>0$ such that for any $j=1,\dots,d$ and $n\geq1$, it holds that $\|R_{j,n}\|_{\mathscr{L}(\cB)}\leq C r_{j,n}$.  
  	\item \label{asmp:ape-ren-spec}(spectral gap): the spectrum of $R\in\mathscr{L}(\cB)$  consists of an isolated simple eigenvalue $1$ and a compact subset of $\bD$. 
	
	\item \label{asmp:ape-ren-ape}(aperiodicity): for each $z\in\bar{\bD}^d\setminus\{(1,\dots,1)\}$ with $|z_1|=\dots=|z_d|=1$, the spectral radius of  $R(z)\in \mathscr{L}(\cB)$ is strictly less than $1$. 
\end{enumerate}
\end{Asmp}

%It is easily seen that $P(R(z)-R)P=\sum_{i,n}\mu[\varphi_i=n](z^n_i-1) P$.
From now on we suppose that Assumption \ref{asmp:ape-ren} is satisfied.
We define a one-dimensional projection $P\in \mathscr{L}(\cB)$ by
\begin{align}
Pu=\int_Y u\:\rd \mu, \quad u\in \cB.
\end{align}
Note that $P$ is the eigenprojection of $R$ for the simple eigenvalue $1$, since $\int_Y 1 \:\rd\mu=1$. It is easily seen that $PR_{j,n}P= r_{j,n}P$.

For $t\geq1$, we denote by $[t]$ the maximal integer which is less than or equal to $t$. Set
\begin{align}\label{def:b_t}
  & b_t= \frac{1}{\Gamma(1-\alpha)\mu[Y\cap \{\varphi\geq[t]\}]}
  \sim
   \frac{t^{\alpha}}{\Gamma(1-\alpha)\ell(t)}, \quad
   \text{as $t\to\infty$.} 
\end{align} 
Note that $(b_n/\Gamma(1+\alpha))_{n=1}^\infty$ is a return sequence in the sense of \cite{Aa97}.

We denote by $\varphi_k$ the $k$th return time to $Y$, that is,
\begin{align}
\varphi_0=0\quad\text{and}\quad
\varphi_k=\sum_{m=1}^{k}\varphi\circ F^{m-1},
\quad k=1,2,\dots
\end{align}
For non-negative integers $n_1\dots,n_d,k$, we define $T_{n_1,\dots,n_d}(k)\in \mathscr{L}(\cB)$ by
\begin{align}
T_{n_1,\dots,n_d}(k)u=R^k\Big(1_{\textstyle\bigcap_{j=1}^d\{S_{A_j}(\varphi_k)=n_j\}}u\Big), \quad u\in\cB.
\end{align}
Note that $T_{n_1,\dots,n_d}(k)$ is the coefficient of $z_1^{n_1}\cdots z_d^{n_d}$ in $R(z)^k$. 
For $n\geq1$ and a bounded continuous function $g:[0,1]^{d}\times [0,\infty)\to\bR$, we define $T_{n,g}\in\mathscr{L}(\cB)$ by 
\begin{align}
\begin{split}
	T_{n,g}u&=
	\sum_{k=1}^n
	R^k \bigg[g\bigg(\bigg(\frac{S_{A_j}(n)}{n}\bigg)_{j=1}^d, \frac{S_Y(n)}{b_n}\bigg) 1_{\textstyle\{\varphi_k=n\}}u\bigg] 
	\\
	&=
	\sum_{\substack{n_1,\dots,n_d,k \\ n_1+\dots+n_d+k=n}} g\Big(\frac{n_1}{n},\dots,\frac{n_d}{n},\frac{k}{b_n}\Big)T_{n_1,\dots,n_d}(k)u,
	\quad
	u\in\cB.
\end{split}
\end{align}

%Then it is easily seen that
%\begin{align}\label{T_ng}
%    \int_A (T_{n,g}u)\rd \mu
%    =
%    \int_{f^{-n}A} g\bigg(\Big(\frac{S_{A_j}(n)}{n}\Big)_{j=1}^d, \frac{S_{Y}(n)}{b_n}\bigg)u\:\rd \mu,
%    \quad 
%    A\in \cA\cap Y, \;u\in\cB. 	
%\end{align}
%%
%%\begin{Lem}
%%  	
%%\end{Lem}

%More precisely, 
%there exist $\lambda(z)\in \bC$, $P(z),Q(z)\in\mathscr{L}(\cB)$ such that the following hold:
%$R(z)$ can be decomposed as
%\begin{align}
%    R(z)= \lambda(z)P(z)+Q(z),	
%    \quad\text{for $z\in \bar{\bD}^d$ sufficiently close to $(1,\dots,1)$},
%\end{align}
%where $P(z)$ is a one-dimensional projection, $P(z)Q(z)=Q(z)P(z)=0$, the  

\section{Main results}\label{sec:main}

%The notation ``$\wto$'' means the weak convergence of finite measures. 
%For a Polish space $H$, let us denote by $\cP(H)$ the space of Borel probability measures on $H$. We endow $\cP(H)$ with the Polish topology of weak convergence. 

\begin{Thm}\label{Thm:conv-of-density}
Suppose that Assumptions \ref{asmp:dyn-sep}, \ref{asmp:reg-var} and \ref{asmp:ape-ren} are satisfied.  Let $g:[0,1]^{d}\times [0,\infty)\to\bR$ be a bounded continuous function. Then it holds that 
\begin{align}
   \frac{\Gamma(\alpha)n}
   {b_n}T_{n,g}
   \to
   \bE\Big[g\big((U_j)_{j=1}^d, W\big)\Big]P,
   \quad \text{in $\mathscr{L}(\cB)$,}
   \label{Thm:conv-of-density-0}	
\end{align}
as $n\to\infty$.
Here \eqref{Thm:conv-of-density-0} means the convergence  with respect to the operator norm $\|\cdot\|_{\mathscr{L}(\cB)}$. 
\end{Thm}

Theorem \ref{Thm:conv-of-density} will be proved in Section \ref{sec:proof}. 
In the case $g\equiv 1$, the convergence \eqref{Thm:conv-of-density-0} was obtained in \cite[Theorem 1.4]{Gou11}. We remark
\begin{equation}
	\frac{\Gamma(\alpha)n}
   {b_n}
   \sim
   \frac{\pi}{\sin(\pi\alpha)}n^{1-\alpha}\ell(n),
   \quad
   \text{as $n\to\infty$}.
\end{equation}

Moreover, we can relax the condition of the initial density $u$ in Theorem \ref{Thm:conv-of-density} following \cite[Section 10]{MelTer}, as we shall see below.
We denote by $L:L^1(X)\to L^1(X)$ the transfer operator of the map $f:X\to X$ with respect to $\mu$. Then $L$ is characterized by the equation
\begin{equation}
\int_X (Lu)v\:\rd\mu=\int_X u(v\circ f)\:\rd \mu,
\quad
u\in L^1(X),\;v\in L^\infty (X).
\end{equation}
It is easily seen that, for $u\in \cB$, 
\begin{equation}
    T_{n,g}u
    =1_Y L^n
    \bigg[g\bigg(\bigg(\frac{S_{A_j}(n)}{n}\bigg)_{j=1}^d, \frac{S_Y(n)}{b_n}\bigg) u 1_Y\bigg].	
\end{equation}
Set
\begin{equation}
   X_{0,0}=Y
   \quad\text{and}\quad
   X_{j,m}
   =
   \bigcap_{k=0}^{m-1}f^{-k}(A_j) \cap f^{-m}(Y),
  \quad
  j=1,\dots,d \;\text{and}\; m\geq1.	
\end{equation}
Then $\mu(X_{j,m})\leq \mu(Y)$.
For $u\in L^1(X)$ and $(j,m)=(0,0)$ or $(j,m)\in \{1,\dots,d\}\times \bN$, set 
\begin{equation}
	u_{j,m}
	=
	u 1_{X_{j,m}}.
\end{equation} 
Note that $L^m u_{j,m}$ is supported on $Y$ and 
\begin{equation}
\int_Y L^m u_{j,m} \:\rd \mu=\int_X u_{j,m}\:\rd \mu=\int_{X_{j,m}}u\:\rd \mu.
\end{equation}
Let 
\begin{equation}\label{B(X)}
	\cB(X)
    =
    \{u\in L^1(X)\cap L^\infty(X)\::\: L^m u_{j,m}\in \cB \;\text{for each $j,m$}\}.
\end{equation}

\begin{Thm}\label{Thm:extend-observable}
Under the setting of Theorem \ref{Thm:conv-of-density},
suppose $u\in \cB(X)$ satisfies 
\begin{equation}\label{Thm:extend-observable-1}
	\sum_{j,m} \|L^m u_{j,m}\|_{\cB}<\infty \quad
    \text{and}
    \quad
    \|L^m u_{j,m}\|_{\cB}=o(m^{-1}), \quad\text{as $m\to\infty$ $(j=1,\dots,d)$.}
\end{equation}
Then it holds that
\begin{equation}
\frac{\Gamma(\alpha)n}
   {b_n} 1_Y L^n
    \bigg[g\bigg(\bigg(\frac{S_{A_j}(n)}{n}\bigg)_{j=1}^d, \frac{S_Y(n)}{b_n}\bigg) u \bigg]
    \to
    \bE\Big[g\big((U_j)_{j=1}^d, W\big)\Big]\int_X u \:\rd \mu,	
    \quad
    \text{in $\cB$,}
\end{equation}
as $n\to\infty$.
\end{Thm}

Theorem \ref{Thm:extend-observable} will be proved in Section \ref{sec:proof2}.
Note that, for $A\in \cA\cap Y$,
\begin{equation}
	\int_A L^n
    \bigg[g\bigg(\bigg(\frac{S_{A_j}(n)}{n}\bigg)_{j=1}^d, \frac{S_Y(n)}{b_n}\bigg) u \bigg]\rd\mu
    =
    \int_{f^{-n}(A)}
    g\bigg(\bigg(\frac{S_{A_j}(n)}{n}\bigg)_{j=1}^d, \frac{S_Y(n)}{b_n}\bigg)u\: \rd\mu.
\end{equation}
The following corollary immediately follows from Theorem \ref{Thm:extend-observable} and the Portmanteau theorem \cite[Theorem 4.25]{Kal02}. 

\begin{Cor}[generalized uniform laws]\label{Cor:conditional}
Under the setting of Theorem \ref{Thm:extend-observable}, we additionally suppose $u\geq0$ and $\int_X u\: \rd \mu=1$.
Let $\nu(\rd x)=u(x)\mu(\rd x)$, which is an absolutely continuous probability measure on $X$. 
Let $A\in \cA\cap Y$ with $\mu[A]>0$.
Then
 it holds that
\begin{equation}
	\nu[f^{-n}(A)]\sim \frac{b_n}{\Gamma(\alpha)n}\mu[A]\sim
	\frac{\sin(\pi\alpha)}{\pi}\frac{1}{n^{1-\alpha}\ell(n)}\mu[A]
\end{equation}
and 
\begin{equation}\label{Cor:conditional-1}
\frac{1}{\nu[f^{-n}(A)]}
    \int_{f^{-n}(A)}  g\bigg(\bigg(\frac{S_{A_j}(n)}{n}\bigg)_{j=1}^d, \frac{S_{Y}(n)}{b_n}\bigg)\rd \nu
   \to
    \bE\Big[g\big((U_j)_{j=1}^d, W\big)\Big],
\end{equation}	
as $n\to\infty$. In addition, for any $t_1,\dots,t_d\in[0,1]$ and $t_{d+1}\geq0$, it holds that
\begin{equation}
\begin{split}\label{Cor:conditional-2}
    &\nu\bigg[\frac{S_{A_1}(n)}{n}\leq t_1,\dots,  \frac{S_{A_d}(n)}{n}\leq t_d,\; 
    \frac{S_{Y}(n)}{b_n}\leq t_{d+1}
    \:\bigg|\:f^{-n}(A)\bigg]
    \\[5pt]
    &\to
    \bP[U_1\leq t_1,\dots,U_d\leq t_d,\; W\leq t_{d+1}],	
\end{split}	
\end{equation}
as $n\to\infty$. Here $\nu[A|B]=\nu[A\cap B]/\nu[B]$ denotes the conditional probability of $A$ given $B$. 
\end{Cor}

%\begin{Rem}
%As shown in Lemma \ref{Thm:conv-of-density} (see also \cite[Theorem 1.4]{Gou11}), it holds that $\nu_n(T^{-n}Y)\underset{n\to\infty}{\sim}\frac{b(n)}{\Gamma(\alpha)n}$  and hence $\nu_n[\cdot|T^{-n}Y]$ is well-defined for sufficiently large $n$.	
%\end{Rem}
%
%Furthermore, we are interested in the weak limit of
%the conditional probability measure
%\begin{align}
%   \nu\bigg[\bigg(\frac{1}{n}S_{A_1}(nt),\dots,\frac{1}{n}S_{A_d}(nt), \frac{1}{b(n)}S_{Y}(nt):t\in[0,1]\bigg)\in\cdot\bigg|T^{-n}Y\bigg]\in \cP(D([0,1],\bR^{d+1})).	
%\end{align}
%We will prove that the limit is the joint-law of the occupation time and local time processes of a skew Bessel diffusion bridge. 

\section{Asymptotic behaviors of $T_{n_1,\dots,n_d}(k)$}
\label{sec:asymp}

In this section we prepare two types of estimates of $T_{n_1,\dots,n_d}(k)$, following \cite{AaDe, Gou11, MelTer}. These estimates are utilized in the proof of Theorem \ref{Thm:conv-of-density}.

\subsection{Local limit theorem and local large deviation for $T_{n_1,\dots,n_d}(k)$}

Set 
\begin{align}\label{def:a_k}
a_k=\min\{n\in\bN:b_n>k\}, \quad k\in \bN. 
\end{align}
Then $b_t\leq k$ if and only if $t< a_k$. 
The sequence $(a_k)_{k=1}^\infty$ is regularly varying with index $1/\alpha$ and 
\begin{equation}\label{eq:reg-var-ak}
\frac{a_k^\alpha}{\Gamma(1-\alpha)\ell(a_k)}\sim k,
\quad \text{as $k\to\infty$},
\end{equation}
which follows from the theory of regular variation \cite[Theorem 1.5.12]{BGT}. We remark that our definition of $a_k$ is asymptotically the same as that of \cite{Gou11} up to multiplication by a constant.

The following proposition is a certain multi-dimensional extension of \cite[Theorem 6.2]{AaDe} and \cite[Proposition 1.5]{Gou11}.

\begin{Prop}[multi-dimensional local limit theorem]\label{Prop:loc-lim}
Suppose that Assumptions \ref{asmp:dyn-sep}, \ref{asmp:reg-var} and \ref{asmp:ape-ren} are satisfied.  Then 
  \begin{align}
  	\sup_{n_1,\dots n_d\in \bZ_{\geq0}}\bigg\|a_k^d T_{n_1,\dots,n_d}(k) - 
  	\bigg(\prod_{j=1}^d \psi_j\Big(\frac{n_j}{a_k}\Big)\bigg)P\bigg\|_{\mathscr{L}(\cB)}
  	\to 0,
  	\quad\text{as $k\to\infty$,} 
  \end{align}
 where $\psi_j:[0,\infty)\to[0,\infty)$ denotes the continuous density function of the one-sided $\alpha$-stable random variable $\xi_j$ in Section \ref{sec:limit}.
\end{Prop}

For the proof of Proposition \ref{Prop:loc-lim}, see Section  \ref{subsec:Proof-LLT}.

By \cite[Theorem 1.6]{Gou11}, we can get a local large deviation for
\begin{align*}
 R^k\Big(1_{\textstyle\{\varphi_k=n\}}u\Big)=\sum_{\substack{n_1,\dots,n_d\\ n_1+\dots+n_d+k=n}}T_{n_1,\dots,n_d}(k)u,
\quad u\in\cB. 
\end{align*}
Moreover we can obtain the following proposition as a slight extension.

\begin{Prop}[multi-dimensional local large deviation]\label{Prop:large-deviation}
	Suppose that Assumptions \ref{asmp:dyn-sep}, \ref{asmp:reg-var} and \ref{asmp:ape-ren} are satisfied. Then there exists some constant $C>0$ such that, for any $n,k\in\bN$ with  $n\geq a_k$ (or equivalently $b_n> k$),
\begin{align}
   \sum_{\substack{n_1,\dots,n_d\\ n_1+\dots+n_d+k=n}}\|T_{n_1,\dots,n_d}(k)\|_{\mathscr{L}(\cB)}\leq   Ckn^{-\alpha-1}\ell(n).	
\end{align}
\end{Prop}

For the proof of Proposition \ref{Prop:large-deviation}, see Section \ref{subsec:Proof-LLD}.

%We omit the proof of Proposition \ref{Prop:large-deviation} since it is almost the same as that of \cite[Theorem 1.6]{Gou11}.
%This kind of estimates is analogous to a one-sided local large deviation for i.i.d.\:partial sums \cite{Don97}.  

\subsection{Proof of Proposition \ref{Prop:loc-lim}}
\label{subsec:Proof-LLT}

For the proof of Proposition \ref{Prop:loc-lim}, we begin with auxiliary results derived from the perturbation theory  \cite{Kat}.
For $z\in\bar{\bD}^d$ close to $(1,\dots,1)$, the spectrum of $R(z)$ consists of an isolated simple eigenvalue $\lambda(z)\in\bC$ close to $1$ and a compact subset of $\bD$.  
To be specific, $R(z)$ can be decomposed as
\begin{align}\label{eq:spec-deco}
	R(z)=\lambda(z)P(z)+Q(z),
	\quad
	\text{for $z\in\bar{\bD}^d$ close to $(1,\dots,1)$,}
\end{align}
where the map $z\mapsto (\lambda(z), P(z), Q(z))$ is continuous, $\lambda(z)\in \bC$ with $\lambda(1,\dots,1)=1$,  the operator $P(z)\in\mathscr{L}(\cB)$ denotes a one-dimensional projection with $P(1,\dots,1)=P$, the spectral radius of $Q(z)\in \mathscr{L}(\cB)$ is less than some constant $\rho\in(0,1)$ uniformly in $z$ close to $(1,\dots,1)$, and 
\begin{equation}
P(z)Q(z)=Q(z)P(z)=0.
\end{equation}
Therefore 
\begin{equation}
\|R(z)^k-\lambda(z)^k P(z)\|_{\mathscr{L}(\cB)}=\|Q(z)^k\|_{\mathscr{L}(\cB)}=O(\rho^k)=o(1),
\quad\text{as $k\to\infty$,}
\end{equation}
 uniformly in $z$ close to $(1,\dots,1)$.    For $t=(t_1,\dots,t_d)\in \bR^d$, set 
\begin{align}
z_t = (e^{it_1},\dots,e^{it_d}).
\end{align}
%The following lemma is a modification of \cite[Theorem 5.1]{AaDe} and \cite[Lemma 3.1 (a)]{MelTer}.

\begin{Lem}\label{Lem:spec-gap}
Let $s=(s_1,\dots,s_d)\in\bR^d$ be fixed. Then, as $k\to\infty$,  
\begin{align}\label{Lem:spec-gap-1}
\lambda(z_{s/a_k})^k \to \prod_{j=1}^d\bE\big[e^{i s_j \xi_j}\big], \;\;\text{in $\bC$,}
\quad 
R(z_{s/a_k})^k \to  \Big(\prod_{j=1}^d\bE\big[e^{i s_j \xi_j}\big]\Big)P,
\;\;\text{in $\mathscr{L}(\cB)$}.
\end{align}
Here it is understood that $s/a_k=(s_1/a_k,\dots,s_d/a_k)$.
\end{Lem}
\begin{proof}
By \cite[Theorem 1.6]{Klo} or a similar discussion as in \cite[VIII, \S2]{Kat}, we have the following asymptotic expansion: as $t\to(0,\dots,0)$,
\begin{align*}
  1-\lambda(z_t) 
  &= 
  \text{\Big(a unique eigenvalue of $P(R-R(z_t))P$\Big)} + O\Big(\|R-R(z_t)\|_{\mathscr{L}(\cB)}^2\Big)
  \\
  &= \sum_{\substack{j=1,\dots,d\\ n\geq1}} (1-e^{int_j})r_{j,{n}} + O\bigg(\bigg\|\sum_{\substack{j=1,\dots,d\\ n\geq1}}(1-e^{int_j})R_{j,n}\bigg\|_{\mathscr{L}(\cB)}^2\bigg).
\end{align*}
Here we used $PR_{j,n}P=r_{j,n}P$. Set 
\begin{align*}
c_\alpha
&=-i\int_0^\infty e^{i\sigma}\sigma^{-\alpha}\rd \sigma
=\Gamma(1-\alpha)(-i)^{\alpha}
\\
&= \Gamma(1-\alpha)
   \bigg(\cos\bigg(\frac{\pi\alpha}{2}\bigg)-i\sin\bigg(\frac{\pi\alpha}{2}\bigg)\bigg).
\end{align*}
 By Assumption \ref{asmp:reg-var} and \cite[Lemma 2.6.1]{IbrLin}, we see that
 \begin{align*}
 	 \sum_{n\geq1}(1-e^{inu})r_{j,n}\sim 
c_\alpha \beta_j u^\alpha \ell(u^{-1}), \quad\text{as $u\to 0+$}.
 \end{align*}
In a similar way, Assumption \ref{asmp:ape-ren} implies  
\begin{align*}
	\|\sum_{n\geq1}(1-e^{inu})R_{j,n}\|_{\mathscr{L}(\cB)}=O(u^\alpha\ell(u^{-1})), \quad\text{as $u\to 0+$}.
\end{align*}
For $s=(s_1,\dots,s_d)\in(0,\infty)^d$, we use \eqref{eq:reg-var-ak} to get
\begin{align*}
   k(1-\lambda(z_{s/a_k}))
   \sim
   k c_\alpha
   \sum_{j=1}^d
   \beta_j (s_j/a_k)^\alpha \ell(a_k/s_j)
   \to \frac{c_\alpha}{\Gamma(1-\alpha)} \sum_{j=1}^d  s_j^\alpha \beta_j,
   \quad \text{as $k\to\infty$},
\end{align*}
which implies 
\begin{align*}
   \lambda(z_{s/a_k})^k
   \to
   \exp\bigg(\frac{-c_\alpha}{\Gamma(1-\alpha)} \sum_{j=1}^d s_j^\alpha \beta_j \bigg)	
   = 
   \prod_{j=1}^d\bE\big[e^{i s_j \xi_j}\big], \quad\text{as $k\to\infty$}.
\end{align*}
Since $\|R(z_{s/a_k})^k-\lambda(z_{s/a_k})^k P\|_{\mathscr{L}(\cB)}\to 0$, we obtain \eqref{Lem:spec-gap-1} in the case $s\in(0,\infty)^d$. We can also deal with the case  $s\notin(0,\infty)^d$ in a similar way.  
\end{proof}

\begin{Lem}\label{Lem:spec-gap-2}
There exist some constants $\delta,C,c>0$ such that, for any $k\geq1$ and $s=(s_1,\dots,s_d)\in(-\delta a_k, \delta a_k)^d$,
\begin{align}
	|\lambda(z_{s/a_k})|^k
	\leq 
	C\prod_{j=1}^d e^{-c|s_j|^{\alpha/2}},
	\quad
	\|R(z_{s/a_k})^k\|_{\mathscr{L}(\cB)}
	\leq 
	C\prod_{j=1}^d e^{-c|s_j|^{\alpha/2}}.
\end{align}	
\end{Lem}

\begin{proof}
We only need to prove the former estimate. By virtue of the Potter bound \cite[Theorem 1.5.6]{BGT}, there exist $\delta,c_1,c_2 >0$ such that, for any $k\geq1$ and  $s\in(-\delta a_k,\delta a_k)^d$, 
\begin{align*}
  k\Re\big(1-\lambda(z_{s/a_k})\big)
  \geq  
   c_1 k a_k^{-\alpha}\ell(a_k) \sum_{j=1}^d |s_j|^{\alpha/2}
   \geq 
   c_2 \sum_{j=1}^d |s_j|^{\alpha/2},	
\end{align*}
which implies the desired estimate.
\end{proof}

We now prove Proposition \ref{Prop:loc-lim} by using the above two lemmas.

\begin{proof}[Proof of Proposition \ref{Prop:loc-lim}]
By the integral formula of Fourier coefficients, we have
\begin{align}
\begin{split}
\label{eq:Fourier-inv}
  T_{n_1,\dots, n_d}(k)
  &=
  \frac{1}{(2\pi)^d}
     \int_{(-\pi, \pi]^d} 
     R(z_s)^k \prod_{j=1}^d e^{-i n_j s_j}
     \rd s_1 \cdots \rd s_d
   \\
   &=
   \frac{1}{(2\pi a_k)^d}
	\int_{(-\pi a_k, \pi a_k]^d}
	R(z_{s/a_k})^k \prod_{j=1}^d e^{-i  s_j n_j/a_k}\rd s_1\cdots \rd s_d.  
\end{split}	
\end{align}
Take $\delta>0$ so small as in Lemma \ref{Lem:spec-gap-2}.
Then Assumption \ref{asmp:ape-ren} implies that there exists $\rho\in(0,1)$ such that
\begin{equation}
\|R(z_{u})^k \|_{\mathscr{L}(\cB)}= O(\rho^k),
\quad
\text{uniformly in $u\in(-\pi,\pi]^d\setminus (-\delta,\delta)^d$, as $k\to\infty$.}
\end{equation} 
By \eqref{eq:Fourier-inv} and \eqref{def:psi_j} we have
\begin{align}\label{eq:loc-lim-1}
\begin{split}
  &\sup_{n_1,\dots n_d\in \bZ_{\geq0}}\bigg\|a_k^d T_{n_1,\dots,n_d}(k) - 
  	\bigg(\prod_{j=1}^d \psi_j\Big(\frac{n_j}{a_k}\Big)\bigg)P\bigg\|_{\mathscr{L}(\cB)}
  \\
  &\leq 
  \frac{1}{(2\pi)^d}\int_{(-\delta a_k,\delta a_k)^d}
  \bigg\|R(z_{s/a_k})^k- \Big(\prod_{j=1}^d\bE\big[e^{i s_j \xi_j}\big]\Big)P\bigg\|_{\mathscr{L}(\cB)}\rd s_1\cdots \rd s_d + o(1),
\end{split}
\end{align}
as $k\to\infty$.
By virtue of Lemmas \ref{Lem:spec-gap} and \ref{Lem:spec-gap-2} and the equality \eqref{eq:Fourier-xi}, we can apply the dominated convergence theorem to the right-hand side of \eqref{eq:loc-lim-1}, which implies the desired result.
\end{proof}

\subsection{Proof of Proposition \ref{Prop:large-deviation}}
\label{subsec:Proof-LLD}

We can prove Proposition \ref{Prop:large-deviation} in almost the same way as in \cite[Theorem 1.6]{Gou11}, as we shall see below.

For $s\in [0,\infty]$ and $z=(z_1,\dots,z_d)\in\bar{\bD}^d$, we define a truncated series $R^{(s)}(z)\in \mathscr{L}(\cB)$ by
\begin{align}
	R^{(s)}(z)u
	%=\Big(R_{0,0} + \sum_{\substack{j=1,\dots,d \\ n<s-1}} z_j^n R_{j, n}\Big) u
	= R(z)\left(1_{\textstyle\{\varphi< s\}}u\right),
	\quad u\in\cB.
\end{align}
If $s<\infty$ then $R^{(s)}(z)$ is a polynomial and hence it is well-defined for $z\in\bC^d$. 
In addition we define $T^{(s)}_{n_1,\dots,n_d}(k)\in\mathscr{L}(\cB)$ as the coefficient of $z_1^{n_1}\dots z_d^{n_d}$ in $R^{(s)}(z)^k$, that is,
  \begin{align}
T_{n_1,\dots,n_d}^{(s)}(k)u
=T_{n_1,\dots,n_d}(k)
\left( 1_{\textstyle\bigcap_{p=1}^{k}\{\varphi\circ F^{p-1} < s\}}u\right), \quad u\in\cB.
\end{align}
It is obvious that $R^{(\infty)}(z)=R(z)$ and $T^{(\infty)}_{n_1,\dots,n_d}(k)=T_{n_1,\dots,n_d}(k)$. Furthermore we write
\begin{align*}
  \|R^{(s)}(z)^k\|_A = \sum_{n_1,\dots,n_d\in\bZ_{\geq0}} \|T^{(s)}_{n_1,\dots,n_d}(k)\|_{\mathscr{L}(\cB)}.	
\end{align*}
 
The following two lemmas are slight extensions of \cite[Lemmas 3.1 and 3.2]{Gou11}. We omit the proofs of them. 

\begin{Lem}\label{Lem:bound-A-norm}
There exists a constant $C>0$ such that, for any $k\in \bN$ and $s\in[a_k/2,\infty]$,
\begin{align*}
   \|R^{(s)}(z)^k\|_A \leq C.	
\end{align*}
\end{Lem}

\begin{Lem}\label{Lem:truncated-local-lim}
There exists a constant $C>0$ such that, for any $k\in\bN$, $s\in [a_k/2,\infty]$ and $n_1,\dots,n_d\in \bZ_{\geq0}$,
\begin{align*}
 \|T^{(s)}_{n_1,\dots,n_d}(k)\|_{\mathscr{L}(\cB)}
 \leq
 \frac{Ce^{-(n_1+\dots+n_d)/s}}{a_k^d}.	
\end{align*}
	
\end{Lem}

\begin{proof}[Proof of Proposition \ref{Prop:large-deviation}]
We only need to focus on the case $k$ is sufficiently large.
Fix a constant $\gamma\in(\frac{1+\alpha}{1+2\alpha},1)$. For $n\geq a_k$, set 
\begin{align*}
w=w(n,k)=\frac{n}{a_k}\geq1 \quad\text{and}\quad\zeta=\zeta(n,k)= \frac{w^\gamma a_k}{2}\in\left[\frac{a_k}{2}, \frac{n}{2}\right].
\end{align*}
We define a Borel subset $\Omega\subset Y$ by 
\begin{align*}
\Omega=\Omega(n,k)=\left\{\sum_{p=1}^{k}\varphi\circ F^{p-1}=n\right\}. 
\end{align*}
We decompose $\Omega$ as $\Omega=\Omega_0\cup\dots \cup \Omega_3$ where
\begin{align*}
   \Omega_3
   &=
   \Omega\cap \bigcup_{p=1}^k\{\varphi\circ F^{p-1}\geq n/2\},
   \\
   \Omega_2
   &=
   \Omega\cap
   \Omega_3^c\cap\bigcup_{1\leq p<q\leq k}\{\varphi\circ F^{p-1}\geq\zeta\;\;\text{and}\;\; \varphi\circ F^{q-1}\geq \zeta\},
   \\
   \Omega_1
   &=
   \Omega\cap\Omega_3^c
   \cap\Omega_2^c \cap
   \bigcup_{p=1}^k
   \{\varphi\circ F^{p-1}\geq\zeta\},
   \\
   \Omega_0
   &=
   \Omega\cap\bigcap_{p=1}^k
   \{\varphi\circ F^{p-1}< \zeta\}.	
\end{align*}
Note the map $\cB \ni u\mapsto T_{n_1,\dots,n_d}(k)(1_{\Omega_i}u)\in \cB$ belongs to $\mathscr{L}(\cB)$.
Define $\Sigma_i$  by
\begin{align*}
   \Sigma_i=\Sigma_i(n,k)=  \sum_{\substack{n_1,\dots,n_d\\ n_1+\dots+n_d+k=n}}
   \sup_{\substack{u\in \cB \\ \|u\|_{\cB}\leq 1}}
   \|T_{n_1,\dots,n_d}(k)(1_{\Omega_i}u)\|_{\cB},
   \quad i=0,1,2,3.	
\end{align*}
We will show that there exists $C>0$ such that $\Sigma_i\leq C k n^{-\alpha-1}\ell(n)$ for $i=0, 1, 2,3$, which immediately implies the desired estimate.

\emph{Estimate of $\Sigma_3$.}
We decompose $\Sigma_3$ according to the value of $p$ such that $\varphi\circ F^{p-1}\geq n/2$ and then we have
\begin{align*}
  &\Sigma_3
  \\
  &\leq  \sum_{p=1}^k \sum_{\substack{j=1,\dots,d,\\ l> n/2}} \sum_{\substack{n_i,m_i\:(i=1,\dots,d)\\ \sum_i (n_i +m_i)+l+k=n}}
  \|T_{n_1,\dots,n_d}(p-1)\|_{\mathscr{L}(\cB)}\cdot
  \|R_{j,l}\|_{\mathscr{L}(\cB)}
  \cdot
  \|T_{m_1,\dots,m_d}(k-p)\|_{\mathscr{L}(\cB)}
  \\
  &\leq d\sum_{p=1}^k \sum_{\substack{n_i,m_i\:(i=1,\dots,d)\\ \sum_i (n_i +m_i)\leq n/2}} 
  \|T_{n_1,\dots,n_d}(p-1)\|_{\mathscr{L}(\cB)}\cdot
  \bigg(
  \sup_{\substack{j=1,\dots,d\\ l> n/2}}\|R_{j,l}\|_{\mathscr{L}(\cB)}
  \bigg)
  \cdot
  \|T_{m_1,\dots,m_d}(k-p)\|_{\mathscr{L}(\cB)}
  \\
  &
  \leq
  \bigg(\sum_{p=1}^k \|R(z)^{p-1}\|_A \cdot\|R(z)^{k-p}\|_A\bigg) C n^{-\alpha-1}\ell(n)\leq C' kn^{-\alpha-1}\ell(n),	
\end{align*}
for some $C, C'>0$. Here we used Lemma \ref{Lem:bound-A-norm}.

\emph{Estimate of $\Sigma_2$.}
In a similar way as in bounding $\Sigma_3$, we decompose $\Sigma_2$ according to the smallest indeces of $p<q$ such that $\varphi\circ F^{p-1}, \varphi\circ F^{q-1}\geq \zeta$, and then we have
\begin{align*}
  \Sigma_2
  &\leq
  \sum_{1\leq p<q\leq k}
   \|R^{(\zeta)}(z)^{p-1}\|_A 
   \cdot 
   \bigg(\sup_{\substack{j=1,\dots,d, \\ l\in (\zeta, n/2]}}\|R_{j,l}\|_{\mathscr{L}(\cB)}\bigg)\cdot
   \|R^{(\zeta)}(z)^{q-p-1}\|_A
   \\
   &\hspace{20mm}
    \cdot
   \bigg(\sum_{\substack{j=1,\dots,d, \\ l\in (\zeta, n/2]}}\|R_{j,l}\|_{\mathscr{L}(\cB)}\bigg)
   \cdot
    \|R^{(n/2)}(z)^{k-q}\|_A
   \\
   &\leq C k^2 \zeta^{-2\alpha-1} \ell(\zeta).
\end{align*}
for some $C>0$. Note that $k\sim a_k^{\alpha}/(\ell(a_k)\Gamma(1-\alpha))$, $\zeta/a_k=w^\gamma/2$ and $n/\zeta=2w^{1-\gamma}$. 
Potter's bounds yield that, for any $\varepsilon>0$, there exists a constant $C'>0$ such that
\begin{align*}
	k\zeta^{-\alpha}\ell(\zeta)\leq C'w^{-\alpha\gamma+\varepsilon}
	\quad
	\text{and}
	\quad
	\frac{\zeta^{-\alpha-1}\ell(\zeta)}{n^{-\alpha-1}\ell(n)}\leq C'w^{(\alpha+1)(1-\gamma)+\varepsilon}
\quad\text{for sufficiently large $k,n$.}
\end{align*}
Here we used $\zeta/a_k=w^\gamma/2$ and $\zeta/n=2w^{-1+\gamma}$.
Therefore 
\begin{align*}
k\zeta^{-2\alpha-1}\ell(\zeta)\leq C''w^{-\alpha\gamma+(\alpha+1)(1-\gamma)+2\varepsilon} n^{-\alpha-1}\ell(n),
\quad\text{for sufficiently large $k,n$}.
\end{align*}
Since $-\alpha\gamma+(\alpha+1)(1-\gamma)<0$, we  conclude that $\Sigma_2\leq C''kn^{-\alpha-1}\ell(n)$ for some $C''>0$ and for sufficiently large $k,n$ with $n\geq a_k$.

\emph{Estimate of $\Sigma_1$.}
We decompose $\Sigma_1$ according to the value $p$ such that $\varphi\circ F^{p-1}\geq\zeta$ and then we have
\begin{align*}
   \Sigma_1
   &\leq 
   \sum_{p=1}^k \sum_{\substack{j=1,\dots,d,\\ l\in(\zeta, n/2]}}\sum_{\substack{n_i,m_i\:(i=1,\dots,d)\\ \sum_i (n_i +m_i)+l+k = n}}  
   \|T_{n_1,\dots,n_d}^{(\zeta)}(p-1)\|_{\mathscr{L}(\cB)}\cdot
  \|R_{j,l}\|_{\mathscr{L}(\cB)}
  \cdot
  \|T_{m_1,\dots,m_d}^{(\zeta)}(k-p)\|_{\mathscr{L}(\cB)}
  \\
  &=:\sum_{p=1}^k \Sigma_1^{(p)}.
\end{align*}
Let us decompose $\sum_p\Sigma_1^{(p)}$ into the sum over $p\leq k/w$ or $p\geq k-k/w$ and the sum over $p\in(k/w, k-k/w)$. Let us fix a constant $\varepsilon\in(0, 1-(\alpha+1)(1-\gamma))$. The former sum is bounded as follows in a similar way as in estimates of $\Sigma_3$ and $\Sigma_2$: there exists constants $C,C', C''>0$ such that, for sufficiently large $k,n$,
\begin{align*}
   \sum_{\substack{p\leq k/w\\\text{or}\; p\geq k-k/w}}\Sigma_1^{(p)}
   &\leq
  \frac{Ck}{w}\sup_{\substack{j=1,\dots,d, \\ l\in (\zeta,n/2]}}\|R_{j,l}\|_{\mathscr{L}(\cB)}
  \leq 
   \frac{C'k}{w}\zeta^{-\alpha-1}\ell(\zeta)
   \\
   &\leq
   C''kw^{-1+(\alpha+1)(1-\gamma)+\varepsilon}n^{-\alpha-1}\ell(n).
\end{align*}
The right-hand side is bounded by $C''k n^{-\alpha-1}\ell(n)$, since the exponent of $w$ is negative.

Let us estimate $\sum_{p\in(k/w, k-k/w)}\Sigma_1^{(p)}$. Without loss of generality we may assume $k<n/6$. Since $\sum_{i}(n_i+m_i)+l+k=n$ and $l<n/2$, we see that $\sum_i n_i>n/6$ or $\sum_i m_i>n/6$. 
Let us focus on the former case.
We use Lemma \ref{Lem:truncated-local-lim} and Potter's bound and then have
\begin{align*}
	\sup_{\substack{p\in(k/w, k-k/w)}} \sup_{\substack{n_1,\dots,n_d\\ \sum_i n_i\in(n/6,n]}}
	\|T_{n_1,\dots,n_d}(p-1)\|_{\mathscr{L}(\cB)}
	&\leq 
	 \sup_{\substack{p\in(k/w, k-k/w)}}\frac{Ce^{-n/(6\zeta)}}{a_{p-1}}
	\\
	&\leq \frac{C'w^{2/\alpha}\exp(-w^{1-\gamma}/3)}{a_k}. 
\end{align*}
for some $C, C'>0$. Hence we see that
\begin{align*}
   &\sum_{p\in(k/w, k-k/w)} \sum_{\substack{j=1,\dots,d,\\ l\in(\zeta, n/2]}}\sum_{\substack{n_i,m_i\:(i=1,\dots,d)\\ \sum_i (n_i +m_i)+l+k = n, \\ \sum_i n_i\in(n/6,n]}}  
   \|T_{n_1,\dots,n_d}^{(\zeta)}(p-1)\|_{\mathscr{L}(\cB)}\cdot
  \|R_{j,l}\|_{\mathscr{L}(\cB)}
  \cdot
  \|T_{m_1,\dots,m_d}^{(\zeta)}(k-p)\|_{\mathscr{L}(\cB)}
  \\
  &\leq 
  k \frac{C'w^{2/\alpha}\exp(-w^{1-\gamma}/3)}{a_k}
  \sum_{\substack{j=1,\dots,d, \\ l\in (\zeta, n/2]}}\sum_{\substack{m_1,\dots,m_d\\ \sum_i m_i+l<5n/6}}  
  \|R_{j,l}\|_{\mathscr{L}(\cB)}
  \cdot
  \|T_{m_1,\dots,m_d}^{(\zeta)}(k-p)\|_{\mathscr{L}(\cB)}
  \\
  &\leq
  k\frac{C'w^{1+2/\alpha}\exp(-w^{1-\gamma}/3)}{n}
 \bigg(\sum_{\substack{j=1,\dots,d, \\ l\in (\zeta, n/2]}}\|R_{j,l}\|_{\mathscr{L}(\cB)}\bigg)
   \cdot
    \|R^{(\zeta)}(z)^{k-p}\|_A
   \\
   &\leq
   C'' kn^{-\alpha-1}\ell(n) w^{C'''} \exp(-w^{1-\gamma}/3).
\end{align*}
for some $C'', C'''>0$ and for sufficiently large $k,n$. Here we used Lemma \ref{Lem:bound-A-norm} and
\begin{align*}
	\sum_{\substack{j=1,\dots,d, \\ l\in (\zeta, n/2]}}\|R_{j,l}\|_{\mathscr{L}(\cB)}\leq C_0\zeta^{-\alpha}\ell(\zeta)\leq C_1 w^{2(1-\gamma)\alpha}n^{-\alpha}\ell(n),
\end{align*}
for some $C_0,C_1>0$ for sufficiently large $k,n$, as in the estimate of  $\Sigma_2$. Therefore we obtain the desired estimate in the case $\sum_i n_i>n/6$. In almost the same way we can deal with the case $\sum_i m_i>n/6$. Therefore we conclude that $\Sigma_1\leq Ckn^{-1-\alpha}\ell(n)$ for some $C>0$.

\emph{Estimate of $\Sigma_0$.}
We use Lemma \ref{Lem:truncated-local-lim} to obtain
\begin{align}\label{estimate-sigma0}
\Sigma_0=\|R^{(\zeta)}(z)^k\|_A
\leq
\sum_{\substack{n_1,\dots,n_d\\ n_1+\dots+n_d+k=n}}\frac{Ce^{-n/\zeta}}{a_k^d}
\leq
\frac{Cw^{d-1}\exp(-2w^{1-\gamma})}{a_k}.
\end{align} 
By Potter's bounds, %for any $\varepsilon>0$, 
there exists $C', C''>0$ such that, for sufficiently large $k,n$,
\begin{align}\label{estimate-sigma0-1}
   kn^{-\alpha-1}\ell(n)\geq C'\frac{n^{-\alpha-1}\ell(n)}{a_k^{-\alpha}\ell(a_k)}
   \geq
   C''\frac{w^{-\alpha-2}}{a_k}.	
\end{align}
Note that $w^{d-1}\exp(-2w^{1-\gamma})=o(w^{-\alpha-2})$, as $w\to\infty$.
Comparing \eqref{estimate-sigma0} and \eqref{estimate-sigma0-1}, we conclude that $\Sigma_0\leq  C'''kn^{-\alpha-1}$ for some $C'''>0$.
\end{proof}

\section{Proof of Theorem \ref{Thm:conv-of-density}}\label{sec:proof}

Recall $b_t$ and $a_k$ are defined by \eqref{def:b_t} and \eqref{def:a_k}, respectively. Let us prepare two auxiliary lemmas.

\begin{Lem}\label{Lem:k-bn}
For any $c>1$,
\begin{equation}
	\sup\bigg\{\bigg|\frac{k}{b_n}-\Big(\frac{a_k}{n}\Big)^{\alpha}\bigg|\::\:k\in [b_{n/c},b_{cn})\bigg\}\to 0,
	\quad \text{as $n\to\infty$}.
\end{equation}	
\end{Lem}

\begin{proof}
Write $b(n)=b_n$.
Then 
\begin{equation}
\frac{b(a_k-1)}{b_n}\leq \frac{k}{b_n} <\frac{b(a_k)}{b_n}.
\end{equation}
Note that $k\in [b_{n/c},b_{cn})$ if and only if $a_k\in(n/c, cn]$.
By the uniform convergence theorem for regular varying functions \cite[Theorem 1.5.2]{BGT},
\begin{equation}
   \sup
   \bigg\{
     \bigg|\frac{b(a_k-1)}{b_n}
           -\Big(\frac{a_k}{n}\Big)^{\alpha}\bigg|
     +\bigg|\frac{b(a_k)}{b_n}
           -\Big(\frac{a_k}{n}\Big)^{\alpha}\bigg|
     \::\:k\in [b_{n/c},b_{cn})
   \bigg\}\to 0,
	\quad \text{as $n\to\infty$},	
\end{equation}
which implies the desired result.
\end{proof}

\begin{Lem}\label{Lem:Lambda-n}
Fix a constant $c>1$. For $n\geq c$, define a finite measure $\Lambda_n$ on $[0,\infty)^d$ by
\begin{equation}
	\Lambda_n(\rd y_1\dots \rd y_d)
	=
	\sum_{k\in [b_{n/c},b_{cn})}\sum_{\substack{n_1,\dots,n_d \\ n_1+\dots+n_d=n-k}}
	\delta_{(n_1/a_k,\dots,n_d/a_k)}
	(\rd y_1\dots \rd y_d),
\end{equation}
where $\delta_{(n_1/a_k,\dots,n_d/a_k)}$ denotes the Dirac measure on $(n_1/a_k,\dots,n_d/a_k)$.
Then 
\begin{equation}
\frac{\Lambda_n(\rd y_1\cdots\rd y_d)}{\alpha b_n n^{d-1}}
	 \to
	 \frac{1_{\textstyle [0,\infty)^d \cap\{c^{-1}< \sum_{j=1}^d y_j < c\}}\rd y_1\cdots\rd y_d}
	      {(\sum_{j=1}^d y_j)^{d+\alpha}}
	     ,
	 \quad 
	 \text{as $n\to\infty$},	
\end{equation}
in the sense of the vague convergence of finite measures on $[0,\infty)^d$.
\end{Lem}

\begin{proof}
It is sufficient to prove convergences of Laplace transforms, as stated in \cite[Theorem 5.3 and Theorem 5.22]{Kal02}. More specifically, we only need to show that, for any $\lambda_1,\dots,\lambda_d\geq0$,
\begin{equation}
\begin{split}\label{conv-Lap}
	&\int_0^\infty 
	\exp\bigg(-\sum_{j=1}^d \lambda_j y_j\bigg)\frac{\Lambda_n(\rd y_1\dots \rd y_d)}{\alpha b_n n^{d-1}}
	\\
    &\to
    \int_{\substack{y_1,\dots,y_d\geq0\\[2pt] c^{-1}< \sum_{j=1}^d y_j < c}}
    \exp\bigg(-\sum_{j=1}^d \lambda_j y_j\bigg)
    \frac{\rd y_1\cdots\rd y_d}{(\sum_{j=1}^d y_j)^{d+\alpha}}
	      ,
	\quad
	\text{as $n\to\infty$.} 
\end{split}
\end{equation}
Let us prove \eqref{conv-Lap}. Note that
\begin{equation}
\begin{split}
    &\sum_{\substack{n_1,\dots,n_d \\ n_1+\dots+n_d=n}}
	\exp\bigg(-\sum_{j=1}^d \lambda_j n_j/a_k\bigg)
	\\
	&=
	\sum_{\substack{n_1,\dots,n_d \\ n_1+\dots+n_d=n\\ n_1\geq k}}
	\exp\bigg(-\sum_{j=1}^d \lambda_j n_j/a_k\bigg)
	+
	\sum_{\substack{n_1,\dots,n_d \\ n_1+\dots+n_d=n\\ n_1< k}}
	\exp\bigg(-\sum_{j=1}^d \lambda_j n_j/a_k\bigg)
	\\
	&=
	\exp(-\lambda_1 k/a_k)\sum_{\substack{n_1,\dots,n_d \\ n_1+\dots+n_d=n-k}}
	\exp\bigg(-\sum_{j=1}^d \lambda_j n_j/a_k\bigg)
	+
	O(k n^{d-2})
	\\
	&=
	\sum_{\substack{n_1,\dots,n_d \\ n_1+\dots+n_d=n-k}}
	\exp\bigg(-\sum_{j=1}^d \lambda_j n_j/a_k\bigg)
	+
	O(b_n n^{d-2}),
\end{split}	
\end{equation}
uniformly in $k$ for $k\in [b_{n/c},b_{cn})$, as $n\to\infty$.
Therefore
\begin{equation}
\begin{split}
		&\int_0^\infty 
	\exp\bigg(-\sum_{j=1}^d \lambda_j y_j\bigg)\frac{\Lambda_n(\rd y_1\dots \rd y_d)}{\alpha b_n n^{d-1}}
	\\
	&
	=
	\frac{1}{\alpha b_n n^{d-1}}\sum_{k\in [b_{n/c},b_{cn})}\sum_{\substack{n_1,\dots,n_d \\ n_1+\dots+n_d=n}}
	\exp\bigg(-\sum_{j=1}^d \lambda_j n_j/a_k\bigg)
	+o(1)
	\\
	&
	=
	\frac{1}{\alpha b_n n^{d-1}}\sum_{k\in [b_{n/c},b_{cn})}\sum_{\substack{n_1,\dots,n_d \\ n_1+\dots+n_d=n}}
	\exp\bigg(-\sum_{j=1}^d \lambda_j \frac{n_j}{n}\bigg(\frac{b_n}{k}\bigg)^{1/\alpha}\bigg)
	+o(1)
	\\
	&\to
	\frac{1}{\alpha}
	\int_{c^{-\alpha}}^{c^\alpha}
	\rd r
	\int_{\substack{\theta_1,\dots,\theta_{d-1}\geq0 \\ \theta_1+\dots+\theta_{d-1}\leq 1}}
	\rd \theta_1\dots \rd\theta_{d-1}
	\exp\bigg(-\bigg(\sum_{j=1}^{d-1}\lambda_j\theta_j +\lambda_d(1-\sum_{j=1}^{d-1}\theta_j)\bigg)r^{-1/\alpha}\bigg),
\end{split}
\end{equation}
as $n\to\infty$. Here we used Lemma \ref{Lem:k-bn} and the convergence of Riemann sums to the corresponding integral. Let us consider the change of variables $y_j=r^{-1/\alpha}\theta_j$ ($j=1,\dots,d-1$) and $y_d=r^{-1/\alpha}(1-\sum_{j=1}^{d-1}\theta_j)$. Then the corresponding Jacobian determinant is
\begin{equation}
  \frac{\partial(y_1,\dots,y_d)}{\partial(r, \theta_1,\dots, \theta_{d-1})}
  =
  (-1)^{d}\alpha^{-1}r^{-(d+\alpha)/\alpha}
  =
  (-1)^{d}\alpha^{-1}\bigg(\sum_{j=1}^d y_j\bigg)^{d+\alpha}.	
\end{equation}
Therefore the change of variables theorem yields
\begin{equation}
\begin{split}
	&
	\frac{1}{\alpha}\int_{c^{-\alpha}}^{c^\alpha}
	\rd r
	\int_{\substack{\theta_1,\dots,\theta_{d-1}\geq0 \\ \theta_1+\dots+\theta_{d-1}\leq 1}}
	\rd \theta_1\dots \rd\theta_{d-1}
	\exp\bigg(-\bigg(\sum_{j=1}^{d-1}\lambda_j\theta_j +\lambda_d(1-\sum_{j=1}^{d-1}\theta_j)\bigg)r^{-1/\alpha}\bigg)
	\\
	&=
	\int_{\substack{y_1,\dots,y_d\geq0\\[2pt] c^{-1}< \sum_{j=1}^d y_j < c}}
	\exp\bigg(-\sum_{j=1}^d \lambda_j y_j\bigg)
    \frac{\rd y_1\cdots\rd y_d}{(\sum_{j=1}^d y_j)^{d+\alpha}}.
\end{split}	
\end{equation}
We now complete the proof of \eqref{conv-Lap}.
\end{proof}

We now prove Theorem \ref{Thm:conv-of-density} by using the local limit theorem and the local large deviation (Propositions \ref{Prop:loc-lim} and \ref{Prop:large-deviation}).

\begin{proof}[Proof of Theorem \ref{Thm:conv-of-density}]
Let $\varepsilon>0$ be an arbitrarily small positive number. Take a constant $c>1$ large enough so that 
\begin{align}\label{Thm:conv-of-density:K}
	\max\{\bP[W \not\in(c^{-\alpha}, c^{\alpha})],\; c^{-d+\alpha},\; c^{-2\alpha} \}<\varepsilon.
\end{align}
For $n\geq c$, set
\begin{align*}
	&J^{(0)}_n=\{(n_1,\dots,n_d,k)\in\bZ_{\geq0}^{d}\times \bN:n_1+\dots+n_d+k=n \;\;\text{and}\;\; b_{cn}\leq k\},
	\\
	&J^{(1)}_n=\{(n_1,\dots,n_d,k)\in\bZ_{\geq0}^{d}\times \bN:n_1+\dots+n_d+k=n \;\;\text{and}\;\; b_{n/c}\leq k < b_{cn}\},
	\\
	&J^{(2)}_n=\{(n_1,\dots,n_d,k)\in\bZ_{\geq0}^{d}\times \bN:n_1+\dots+n_d+k=n \;\;\text{and}\;\; k< b_{n/c} \}.	
\end{align*}
We decompose $T_{n,g}$ as $T_{n,g}=\sum_{i=0,1,2}T^{(i)}_{n,g}$ where
\begin{align*}
	T^{(i)}_{n,g}=\sum_{(n_1,\dots,n_d,k)\in J^{(i)}_n} g\Big(\frac{n_1}{n},\dots,\frac{n_d}{n},\frac{k}{b_n}\Big)T_{n_1,\dots,n_d}(k),
	\quad i=0,1,2.
\end{align*}
In the following we show that $(\Gamma(\alpha)n/b_n)T^{(1)}_{n,g}$ is approximately equal to $\bE \big[ g\big((U_j)_{j=1}^d, W \big) \big]P$, and that the contributions of $T^{(0)}_{n,g}$ and $T^{(2)}_{n,g}$ are negligibly small.

\emph{Estimate of $T^{(1)}_{n,g}$.}
Let us estimate $T^{(1)}_{n,g}$.  Proposition \ref{Prop:loc-lim} yields that
\begin{align}
	\sup_{(n_1,\dots,n_d,k)\in J^{(1)}_n}  \bigg\|a_k^d T_{n_1,\dots,n_d}(k)
	-\bigg(\prod_{j=1}^d \psi_j\Big(\frac{n_j}{a_k}\Big)\bigg)P\bigg\|_{\mathscr{L}(\cB)}\to 0,
	\quad
	\text{as $n\to\infty$}.
\end{align}
 Hence we have
\begin{align}
\begin{split}
   &\bigg\|T^{(1)}_{n,g}-\sum_{(n_1,\dots,n_d,k)\in J^{(1)}_n} g\Big(\frac{n_1}{n},\dots,\frac{n_d}{n} ,\frac{k}{b_n}\Big) \frac{1}{a_k^d}\bigg(\prod_{j=1}^d \psi_j\Big(\frac{n_j}{a_k}\Big)\bigg)P\bigg\|_{\mathscr{L}(\cB)}
\\
 &=o\Big(\sum_{(n_1,\dots,n_d,k)\in J^{(1)}_n}\frac{1}{a_k^d}\Big),
     \quad\text{as $n\to\infty$}.     
\label{Thm:conv-of-density-2}
\end{split}
\end{align}
Recall $b_t\leq k$ if and only if $t< a_k$. Hence the relation $b_{n/c}\leq k < b_{cn}$ is equivalent to $n/c<a_{k}\leq cn$, which implies
\begin{equation}
	o\bigg(\sum_{(n_1,\dots,n_d,k)\in J^{(1)}_n}\frac{1}{a_k^d}\bigg)
	=
	o\bigg(\sum_{\substack{n_1,\dots,n_{d-1}\leq c a_k \\ b_{n/c}\leq k < b_{cn}}}
	\frac{1}{a_k^d}\bigg)
	=o\bigg(
	\sum_{b_{n/c}\leq k < b_{cn}}\frac{1}{a_k}
	\bigg),\quad
	\text{as $n\to\infty$.}
\end{equation}
By using $\sum_{k>n}a_k^{-1}
 = O(na_n^{-1})$ and $a_{b_n}\sim n$, as $n\to\infty$, we obtain
\begin{equation}
	 o\bigg(
	\sum_{b_{n/c}\leq k < b_{cn}}\frac{1}{a_k}
	\bigg)
	=
	o\bigg(\frac{b_n}{a_{b_n}}\bigg)
	=
	o\bigg(\frac{b_n}{n}\bigg),
	\quad
	\text{as $n\to\infty$.}
\end{equation}

Define a finite measure  $\Lambda_n$ on $[0,\infty)^d$ as in Lemma \ref{Lem:Lambda-n}. 
By Lemmas \ref{Lem:k-bn} and \ref{Lem:Lambda-n} and Definition \ref{DEF:gen-unif-dist}, we see
\begin{align}
\begin{split}
  &\sum_{(n_1,\dots,n_d,k)\in J^{(1)}_n} 
  g\Big(\frac{n_1}{n},\dots,\frac{n_d}{n} ,\frac{k}{b_n}\Big)
  \frac{1}{a_k^d}\prod_{j=1}^d \psi_j\Big(\frac{n_j}{a_k}\Big)
  %%%%%%%%%%%%%%%%%%%%%%%%%%
  \\
  &\sim
  \sum_{(n_1,\dots,n_d,k)\in J^{(1)}_n}
  g\Big(\frac{n_1}{n-k},\dots,\frac{n_d}{n-k} ,\Big(\frac{a_k}{n}\Big)^\alpha\:\Big)
  \Big(\frac{n-k}{na_k}\Big)^d\prod_{j=1}^d \psi_j\Big(\frac{n_j}{a_k}\Big)
  \\
  &=
%  %%%%%%%%%%%%%%%%%%%%%%%%%%
%  \sum_{(n_1,\dots,n_d,k)\in J^{(1)}_n} 
%  g\Big(\frac{n_1}{n},\dots,\frac{n_d}{n} ,\Big(\frac{a_k}{n}\Big)^\alpha \Big)
%  \frac{1}{(b^{-1}(k))^d}\prod_{i=1}^d \psi_i\Big(\frac{n_i}{b^{-1}(k)}\Big)
%  %%%%%%%%%%%%%%%%%%%%%%%%%%
%  \\
%  &\sim
  %%%%%%%%%%%%%%%%%%%%%%%%%%
  \int_{[0,\infty)^d} \Lambda_n(\rd y_1\cdots\rd y_{d})\:
  g\Big(\frac{y_1}{\sum_{j} y_j},\dots,\frac{y_d}{\sum_{j} y_j},\frac{1}{(\sum_{j} y_j)^\alpha}\Big)\Big(\frac{\sum_{j} y_j}{n}\Big)^d 
  \prod_{j=1}^d \psi_j(y_j)
  \\
  &\sim
	%%%%%%%%%%%%%%%%%%%%%%%%%%%%%%
	\frac{\alpha b_n}{n}%\underset{\sum_{i=1}^d y_i\in[K^{-1},K]}
	\int_{\substack{y_1,\dots,y_d\geq0\\[2pt] c^{-1}< \sum_{j=1}^d y_j < c}}
	 g\Big(\frac{y_1}{\sum_j y_j},\dots,\frac{y_d}{\sum_j y_j},\frac{1}{(\sum_j y_j)^\alpha}\Big)
	\frac{\prod_j \psi_j(y_j)}{(\sum_j y_j)^\alpha}
	\rd y_1\cdots\rd y_d
	%%%%%%%%%%%%%%%%%%%%%%%%%%%%%%
	\\[3pt]
	&=
	%%%%%%%%%%%%%%%%%%%%%%%%%%%%%%
	\frac{b_n}{\Gamma(\alpha)n} \bE\Big[g\big(U_1,\dots,U_d, W\big)1_{\{c^{-\alpha}< W < c^\alpha\}}\Big],
	\label{Thm:conv-of-density-4}
\end{split}
\end{align}
as $n\to\infty$. 
Since $\bP[W\not\in(c^{-\alpha},c^{\alpha})]<\varepsilon$, we now obtain
\begin{align}
   \limsup_{n\to\infty} \bigg\|
	\frac{\Gamma(\alpha)n}{b_n}T^{(1)}_{n,g}-
	\bE\Big[g\big((U_j)_{j=1}^d, W\big)\Big]P
	\bigg\|_{\mathscr{L}(\cB)}\leq (\sup|g|)\varepsilon.
	\label{estimate-T1}
\end{align}

\emph{Estimate of $T^{(0)}_{n,g}$.}  Next, let us show that the contribution of $T^{(0)}_{n,g}$ is negligibly small.
In a similar way as in the estimate of $T^{(1)}_{n,g}$, we have
\begin{align}
	\|T^{(0)}_{n,g}\|_{\mathscr{L}(\cB)}
	&\leq 
	   \sum_{(n_1,\dots,n_d,k)\in J^{(0)}_n} \frac{(\sup|g|)(\prod_j \sup\psi_j)}{a_k^d}
	    +
	    o\Big(\frac{b_n}{n}\Big),
	\quad
	\text{as $n\to\infty$.}
\end{align}
It is easily check that
\begin{align}
   \sum_{(n_1,\dots,n_d,k)\in J^{(0)}_n} \frac{1}{a_k^d}
   \leq
  \sum_{\substack{n_1,\dots,n_{d-1}\leq c^{-1}a_k \\ b_{n/c}\leq k < b_{cn}}}
	\frac{1}{a_k^d}
   \leq c^{-d+1}\sum_{k>b_{cn}}\frac{1}{a_k}
   \sim \frac{c^{-d+\alpha}}{\alpha^{-1}-1}\frac{b_n}{n},
   \quad
   \text{as $n\to\infty$.}	
\end{align}
Since $c^{-d+\alpha}<\varepsilon$, we obtain
\begin{align}
   \limsup_{n\to\infty}\frac{n}{b_n}\|T^{(0)}_{n,g}\|_{\mathscr{L}(\cB)}
   \leq\Big(\frac{(\sup|g|)(\prod_j \sup\psi_j)}{\alpha^{-1}-1}\Big)\varepsilon.
	\label{estimate-T0}	
\end{align}

\emph{Estimate of $T^{(2)}_{n,g}$.}  Let us show that the contribution of $T^{(2)}_{n,g}$ is also negligibly small. Note that 
\begin{align}
	\|T^{(2)}_{n,g}\|_{\mathscr{L}(\cB)}
	\leq (\sup|g|) \sum_{k\leq b_{n/c}}\sum_{\substack{n_1,\dots,n_d\\ n_1+\dots+n_d+k=n}}\|T_{n_1,\dots,n_d}(k)\|_{\mathscr{L}(\cB)}. 
\end{align}
By Proposition \ref{Prop:large-deviation}, we can take a constant $C>0$ independently of the choice of $c>1$ such that
\begin{align}
    \sum_{k\leq b_{n/c}}\sum_{\substack{n_1,\dots,n_d\\ n_1+\dots+n_d+k=n}}\|T_{n_1,\dots,n_d}(k)\|_{\mathscr{L}(\cB)}
    &\leq
     \frac{C}{nb_n}\sum_{k\leq b_{n/c}}k
    \leq 
    \frac{Cb_{n/c}^2}{2nb_n}
    \sim
    \frac{C c^{-2\alpha} b_n}{2n}.	
\end{align}
Since $c^{-2\alpha}<\varepsilon$, we obtain
\begin{align}
    \limsup_{n\to\infty}\frac{n}{b_n}\|T^{(2)}_{n,g}\|_{\mathscr{L}(\cB)}
	\leq
	\frac{(\sup|g|)C}{2}\varepsilon.
	\label{estimate-T2}
\end{align}

\emph{Estimate of $T_{n,g}$.}
By \eqref{estimate-T1}, \eqref{estimate-T0} and \eqref{estimate-T2}, we can take a constant $C'>0$ independently of the choice of $\varepsilon>0$ so that 
\begin{align}
   \limsup_{n\to\infty}
   \bigg\|
   \frac{\Gamma(\alpha)n}{b_n}T_{n,g}
   -
   \bE\big[g\big((U_j)_{j=1}^d, W\big)\big]P
   \bigg\|_{\mathscr{L}(\cB)}
   \leq C'\varepsilon.
\end{align}
Since $\varepsilon>0$ was arbitrarily chosen, we finally obtain the desired convergence \eqref{Thm:conv-of-density-0}.
\end{proof}

\section{Proof of Theorem \ref{Thm:extend-observable}}
\label{sec:proof2}

Set
\begin{equation}
	w(n)
	:=
	\frac{\Gamma(\alpha)n}
   {b_n}
   \sim
   \frac{\pi}{\sin(\pi\alpha)}n^{1-\alpha}\ell(n),
   \quad
   \text{as $n\to\infty$}.
\end{equation}

\begin{Lem}\label{Lem:extend-1}
For any $\varepsilon>0$, there exist $\delta\in(0,1/2)$ such that, for any $j=1,\dots,d$ and for any positive integers $n\geq 1/\delta$ and $m\leq \delta n$,
\begin{equation}
\begin{split}
	&w(n)\bigg\|
	1_Y L^n
    \bigg[
    g
    \bigg(\frac{S_{A_1}(n)}{n},\dots, \frac{S_{A_d}(n)}{n},  \frac{S_Y(n)}{b_n}
    \bigg) 
    u_{j,m} 
    \bigg]
    -
    T_{n-m,g}(L^m u_{j,m})\bigg\|_{\cB}
    \\
    &\leq 
    \varepsilon \| L^m u_{j,m} \|_{\cB}.
\end{split}
\end{equation}
\end{Lem}

\begin{proof}
It is sufficient to consider the case $j=1$. 
For $n\geq m\geq1$,
\begin{equation}\label{Lem:extend-1-1}
\begin{split}
  &1_YL^n
   \bigg[g\bigg(\frac{S_{A_1}(n)}{n},\dots, \frac{S_{A_d}(n)}{n},  \frac{S_Y(n)}{b_n}\bigg) u_{1,m} \bigg]
  \\
  &=
  1_YL^{n-m}
  \bigg[
  g\bigg(\frac{S_{A_1}(n-m)+m-1}{n}, \bigg(\frac{S_{A_i}(n-m)}{n}\bigg)_{i=2}^d, \frac{S_Y(n-m)+1}{b_n}\bigg)
  L^m u_{1,m}
  \bigg]
  \\
  &=
  \sum_{\substack{n_1,\dots,n_d,k \\ n_1+\dots+n_d+k=n-m}}
  g
  \bigg(\frac{n_1+m-1}{n}, \frac{n_2}{n},\dots, \frac{n_{d}}{n}, \frac{k+1}{b_n}
   \bigg)
      T_{n_1,\dots,n_d}(k)
      (L^m u_{1,m})
  ,
\end{split}
\end{equation}
and for $n>m\geq1$,
\begin{equation}
\begin{split}
	 &T_{n-m,g}(L^m u_{1,m})
	 \\
	 &=
  \sum_{\substack{n_1,\dots,n_d,k \\ n_1+\dots+n_d+k=n-m}}
  g
  \bigg(\frac{n_1}{n-m},\dots, \frac{n_{d}}{n-m}, \frac{k}{b_{n-m}}
   \bigg)
      T_{n_1,\dots,n_d}(k)
      (L^m u_{1,m}).
\end{split}	
\end{equation}

Fix $\varepsilon >0$ arbitrarily. As in the proof of Theorem \ref{Thm:conv-of-density}, we can take $c>1$ large enough so that,
\begin{equation}
	\limsup_{n\to\infty}
	\bigg(w(n)\sum_{k< b_{cn}}
   \sum_{\substack{n_1,\dots,n_d \\ n_1+\dots+n_d+k=n}}
   \| T_{n_1,\dots,n_d}(k)\|_{\mathscr{L}(\cB)}\bigg)
   <
   1+\varepsilon,
\end{equation}
and,
\begin{equation}
\sup_{n}\bigg(w(n)\sum_{k\geq b_{cn}}
\sum_{\substack{n_1,\dots,n_d \\ n_1+\dots+n_d+k=n}}\sup|g|
   \| T_{n_1,\dots,n_d}(k)\|_{\mathscr{L}(\cB)}\bigg)
   < 
   \varepsilon. 	
\end{equation}
Take $\delta_0\in(0,1/2)$ small enough so that, if $(x_1,\dots,x_{d+1}), (y_1,\dots,y_{d+1})\in[0,1]^d\times [0,\infty)$ satisfies
\begin{equation}
  \max_{j=1,\dots,d+1}|x_j-y_j|\leq \delta_0
  \quad\text{and}\quad
  \max\{x_{d+1}, y_{d+1}\}\leq \sup_{n}\frac{b_{cn}}{b_{n/2}},	
\end{equation}
then
\begin{equation}
	|g(x_1,\dots,x_{d+1})-g(y_1,\dots,y_{d+1})|\leq \varepsilon.
\end{equation}
For any positive integers $n\geq 1/\delta_0$ and $ m\leq \delta_0 n$ and non-negative integer $n'\leq n-m-1$,
\begin{equation}
 	0\leq
 	\frac{n'+m-1}{n}-\frac{n'}{n-m}
 	\bigg(=
 	\frac{m(n-m-n')}{n(n-m)}-\frac{1}{n}\leq\frac{m}{n}\bigg)\leq\delta_0
\end{equation}
and
\begin{equation}
	0\geq \frac{n'}{n}-\frac{n'}{n-m}
	 \bigg(= -\frac{mn'}{n(n-m)}\geq -\frac{m}{n}\bigg) 
	\geq -\delta_0.
\end{equation}

By the uniform convergence theorem for regular varying functions, we can take $\delta\in(0, \delta_0)$ small enough so that, for any non-negative integers $n\geq 1/\delta$, $k\leq b_{cn}-1$ and $1\leq m\leq \delta n $, it holds that
\begin{equation}
  \bigg|\frac{k+1}{b_n}-\frac{k}{b_{n-m}}\bigg|\leq \delta_0,
   \quad
   \bigg|1-\frac{w(n)}{w(n-m)}\bigg|\leq \varepsilon,	
\end{equation}
and
\begin{equation}
	\bigg(w(n)\sum_{k< b_{c(n-m)}}
   \sum_{\substack{n_1,\dots,n_d \\ n_1+\dots+n_d+k=n-m}}
   \| T_{n_1,\dots,n_d}(k)\|_{\mathscr{L}(\cB)}\bigg)
   \leq
   1+\varepsilon.
\end{equation}
Hence
\begin{equation}
\begin{split}
	&\bigg\|
	1_Y L^n
    \bigg[
    g
    \bigg(\frac{S_{A_1}(n)}{n},\dots, \frac{S_{A_d}(n)}{n},  \frac{S_Y(n)}{b_n}
    \bigg) 
    u_{j,m} 
    \bigg]
    -
    T_{n-m,g}(L^m u_{j,m})
    \bigg\|_{\cB}
    \\[5pt]
    &\leq
    \varepsilon
    \sum_{k<b_{c(n-m)}}
    \sum_{\substack{n_1,\dots,n_d \\ n_1+\dots+n_d+k=n-m}}
     \|T_{n_1,\dots,n_d}(k)\|_{\mathscr{L}(\cB)}
     \|L^m u_{1,m}\|_{\cB}
    \\[5pt]
    &\hspace{15pt}+
    2\sum_{k\geq b_{c(n-m)}}
    \sum_{\substack{n_1,\dots,n_d \\ n_1+\dots+n_d+k=n-m}}
    \sup|g|\| T_{n_1,\dots,n_d}(k)\|_{\mathscr{L}(\cB)}
    \|L^m u_{1,m}\|_{\cB}
    \\
    &\leq
    \frac{\varepsilon(3+\varepsilon)}{w(n-m)}\|L^m u_{1,m}\|_{\cB}
    \leq
    \frac{\varepsilon(3+\varepsilon)(1+\varepsilon)}{w(n)}\|L^m u_{1,m}\|_{\cB}.
\end{split}
\end{equation}
Since $\varepsilon>0$ was arbitrary, we complete the proof.
\end{proof}

\begin{Lem}\label{Lem:extend-2}
For any $\delta\in(0, 1)$ and $j=1,\dots,d$,
\begin{equation}
	\lim_{n\to\infty}\sum_{1\leq m\leq \delta n}
	\bigg\|w(n)T_{n-m,g}(L^m u_{j,m})
	-
	\bE[g(U_1,\dots,U_d, W)]\int_{X}u_{j,m}\:\rd\mu\bigg\|_{\cB}=0.
\end{equation}	
\end{Lem}

\begin{proof}
Since $\int_X u_{j,m}\:\rd\mu=\int_Y L^m u_{j,m}\:\rd\mu$,
\begin{equation}\label{Lem:extend-2-1}
\begin{split}
    &\sum_{1\leq m\leq \delta n}
	\bigg\|w(n)T_{n-m,g}(L^m u_{j,m})
	-
	\bE[g(U_1,\dots,U_d, W)]\int_{X} u_{j,m}\:\rd\mu\bigg\|_{\cB}
	\\
	&\leq
	\sum_{1\leq m\leq \delta n}
	\|w(n)T_{n-m,g}
	-
	\bE[g(U_1,\dots,U_d, W)]P
	\|_{\mathscr{L}(\cB)}
	\|L^m u_{j,m}\|_{\cB}.
\end{split}	
\end{equation}
For each fixed $m$,
\begin{equation}
	\lim_{n\to\infty}
    \frac{w(n)}{w(n-m)}
	=1,
\end{equation}
and hence, by Theorem \ref{Thm:conv-of-density},
\begin{equation}
\begin{split}
	&\lim_{n\to\infty}\|w(n)T_{n-m,g}
	-
	\bE[g(U_1,\dots,U_d, W)]P\|_{\mathscr{L}(\cB)}
	\\
	&
	=\lim_{n\to\infty}\|w(n-m)T_{n-m,g}
	-
	\bE[g(U_1,\dots,U_d, W)]P\|_{\mathscr{L}(\cB)}
	=0.
\end{split}
\end{equation}
By the uniform convergence theorem for regular varying functions,
\begin{equation}
 C=\sup_{\substack{n,m \\ 1\leq m \leq \delta n}}\frac{w(n)}{w(n-m)}
 <\infty,	
\end{equation}
and hence
\begin{equation}
\begin{split} 
   &\sup_{\substack{n,m \\ 1\leq m \leq \delta n}}\| w(n)T_{n-m,g}
	-
	\bE[g(U_1,\dots,U_d, W)]P\|_{\mathscr{L}(B)}
   \\
   &\leq
   C\sup_{\substack{n,m \\ 1\leq m \leq \delta n}}\|w(n-m)T_{n-m,g}\|_{\mathscr{L}(B)}+\sup|g|<\infty.
\end{split}		
\end{equation}
Since $\sum_{m=1}^\infty\|L^m u_{j,m}\|<\infty$, we can apply the dominated convergence theorem to the right-hand side of \eqref{Lem:extend-2-1}, and then we obtain the desired result.
\end{proof}

\begin{Lem}\label{Lem:extend-3}
For any $\delta\in(0,1)$ and $j=1,\dots,d$,
\begin{equation}
   \lim_{n\to\infty}\bigg(w(n)\sum_{\delta n< m \leq n}
   \bigg\|
	1_Y L^n
    \bigg[
    g
    \bigg(\frac{S_{A_1}(n)}{n},\dots, \frac{S_{A_d}(n)}{n},  \frac{S_Y(n)}{b_n}
    \bigg) 
    u_{j,m} 
    \bigg]
    \bigg\|_{\cB}\bigg)
    =0.
\end{equation}
\end{Lem}

\begin{proof}
By almost the same calculation as in \eqref{Lem:extend-1-1},
\begin{equation}\label{Lem:extend-3-1}
\begin{split}
	&w(n)\sum_{\delta n< m\leq n}
   \bigg\|
	1_Y L^n
    \bigg[
    g
    \bigg(\frac{S_{A_1}(n)}{n},\dots, \frac{S_{A_d}(n)}{n},  \frac{S_Y(n)}{b_n}
    \bigg) 
    u_{j,m} 
    \bigg]
    \bigg\|_{\cB}
    \\
    &\leq
    (\sup|g|)
     w(n)
    \sum_{\delta n< m\leq n}
    \sum_{\substack{n_1,\dots,n_d,k \\ n_1+\dots+n_d+k=n-m}}\|T_{n_1,\dots,n_d}(k)\|_{\mathscr{L}(\cB)}
    \|L^m u_{j,m}\|_{\cB}
    \\
    &\leq
    (\sup|g|)
     w(n)
    \bigg(\sum_{\delta n < m\leq n}\sum_{\substack{n_1,\dots,n_d,k \\ n_1+\dots+n_d+k=n-m}}\|T_{n_1,\dots,n_d}(k)\|_{\mathscr{L}(\cB)}\bigg)
     \frac{1}{\delta n}
     \sup_{m > \delta n} \|m L^m u_{j,m}\|_{\cB}.
\end{split}
\end{equation}
By a similar way as in the proof of Theorem \ref{Thm:conv-of-density}, we see that
\begin{equation}
	\sum_{\substack{n_1,\dots,n_d,k \\ n_1+\dots+n_d+k=n}}\|T_{n_1,\dots,n_d}(k)\|_{\mathscr{L}(\cB)}
	\sim
	\frac{1}{w(n)},
	\quad\text{as $n\to\infty$.}
\end{equation}
Since $1/w(n)$ is regularly varying with index $-1+\alpha$, we have
\begin{equation}
\begin{split}
	\sum_{\delta n < m\leq n}\sum_{\substack{n_1,\dots,n_d,k \\ n_1+\dots+n_d+k=n-m}}\|T_{n_1,\dots,n_d}(k)\|_{\mathscr{L}(\cB)}
	&\leq
	\sum_{m=0}^n\sum_{\substack{n_1,\dots,n_d,k \\ n_1+\dots+n_d+k=m}}\|T_{n_1,\dots,n_d}(k)\|_{\mathscr{L}(\cB)}
	\\
	&\sim
	\frac{n}{\alpha w(n)},
	\quad
	\text{as $n\to\infty$.} 
\end{split}
\end{equation}
Therefore there exists a constant $C>0$ such that
\begin{equation}
\begin{split}
	\text{(R.H.S.\ of \eqref{Lem:extend-3-1})}
	\leq
	C \sup_{m > \delta n} \|m L^m u_{j,m}\|_{\cB}
	\to
	0,
	\quad\text{as $n\to\infty$.}
\end{split}
\end{equation}
Here we used $ \|L^m u_{j,m}\|_{\cB}=o(m^{-1})$, as $m\to\infty$. We now complete the proof.
\end{proof}

We now prove Theorem \ref{Thm:extend-observable} by using Lemmas \ref{Lem:extend-1}, \ref{Lem:extend-2} and \ref{Lem:extend-3}.

\begin{proof}[Proof of Theorem \ref{Thm:extend-observable}]
By Lemmas  \ref{Lem:extend-1} and \ref{Lem:extend-2},
\begin{equation}
\begin{split}
 \lim_{\delta\to 0+}\limsup_{n\to\infty}
 \bigg\|
 w(n)1_YL^n
   \bigg[g\bigg(\frac{S_{A_1}(n)}{n},\dots, \frac{S_{A_d}(n)}{n},  \frac{S_Y(n)}{b_n}\bigg) \sum_{\substack{j,m \\ m \leq \delta n}}u_{j,m} &\bigg]
 \\
 -\bE[g(U_1,\dots,U_d, W)]
 \int_{X}
 \sum_{\substack{j,m \\ m \leq \delta n}}
 u_{j,m}\:\rd\mu
   &\bigg\|_{\cB}=0.		
\end{split}
\end{equation}
By Lemma \ref{Lem:extend-3}, for any $0<\delta<1$,
\begin{equation}
   \lim_{n\to\infty}
   \bigg\|
   w(n)
	1_Y L^n
    \bigg[
    g
    \bigg(\frac{S_{A_1}(n)}{n},\dots, \frac{S_{A_d}(n)}{n},  \frac{S_Y(n)}{b_n}
    \bigg) 
     \sum_{\substack{j,m \\ m > \delta n}}
    u_{j,m} 
    \bigg]
    \bigg\|_{\cB}
    =0.
\end{equation}
In addition, since $u=\sum_{j,m}u_{j,m}\in L^1(X)$,
\begin{equation}
   \lim_{n\to\infty}\int_X \sum_{\substack{j,m \\ m > \delta n}}u_{j,m}\: \rd\mu=0.	
\end{equation}
Combining these facts, we obtain the desired result.	
\end{proof}

\section{Application to intermittent interval maps}
\label{sec:intermittent}

In this section we apply Theorems \ref{Thm:conv-of-density} and \ref{Thm:extend-observable} to a certain class of intermittent interval maps. For simplicity, we only deal with intermittent interval maps which have finite number of branches and an indifferent fixed point on each branch.

\begin{Asmp}[Thaler's map \cite{Tha80, Tha83}]\label{asmp:Thaler}
For an integer $d\geq2$, let $(J_j:j=1,\dots,d)$ be a finite collection of disjoint subintervals of $[0,1]$ such that $\sup J_j=\inf J_{j+1}$, $j=1,\dots,d-1$, and $\bigcup_{j=1}^d \bar{J_j}=[0,1]$.
Suppose that an interval map $f:[0,1]\to[0,1]$ satisfies the following:
\begin{enumerate}
\item For each $j$, the restriction $f|_{J_j}$ can be extended to a $C^2$-bijection $f_j:\bar{J_j}\to[0,1]$.
\item For each $j$, there exists an indifferent fixed point $x_j\in J_j$, that is, $f(x_j)=x_j$ and $f'(x_j)=1$. In addition there exists $\eta>0$ such that $f'$ is decreasing on $(x_j-\eta,x_j)$ and increasing on $(x_j,x_j+\eta)$. 
\item For any $\varepsilon>0$, it holds that $\inf\{f'(x):x\in\bigcup_{j}J_j\setminus (x_j-\varepsilon, x_j+\varepsilon)\}>1$. 
\item There exist $\alpha\in(0,1)$, $c_j>0$ and a measurable function $\ell^*:(0,\infty)\to(0,\infty)$ slowly varying at $0$ such that\begin{align}
|f(x)-x|\sim c_j|x-x_j|^{1+1/\alpha}\ell^*(|x-x_j|), \quad\text{as $x\to x_j$ $(j=1,\dots,d)$.}	
\end{align}	
\end{enumerate}
\end{Asmp}

In this section we always suppose Assumption \ref{asmp:Thaler} is satisfied. 
Note that $x_1=0$ and $x_d=1$.
The map $f$ admits an invariant density $h(x)$ which is continuous and positive on $[0,1]\setminus\{x_1,\dots,x_d\}$. For some continuous function $h_0:[0,1]\to(0,\infty)$ and some constant $c'_1,\dots,c'_d>0$, the invariant density $h(x)$ can be written as 
\begin{equation}
\begin{split}
	h(x)
	&=
	h_0(x)
	\prod_{j=1,\dots,d}
	\frac{x-x_j}{x-f_j^{-1}(x)},
	\quad
	x\in[0,1]\setminus\{x_1,\dots, x_d\},
	\\
	&\sim
	\frac{c_j'}
	{|x-x_j|^{1/\alpha}\ell^*(|x-x_j|)},
	\quad
	\text{as $x\to x_j$ $(j=1,\dots,d)$.}	
\end{split}
\end{equation}
We denote by $\mu(\rd x)=h(x)\rd x$ an absolutely continuous $f$-invariant measure on $[0,1]$, which has infinite measure on any neighborhood of $x_j$. The map $f$ is  a CEMPT with respect to $\mu$.

We denote by $\operatorname{Lip}([0,1])$ the space of  functions $v:[0,1]\to\bC$ which is Lipschitz continuous with respect to the Euclidean norm.

\begin{Thm}\label{Thm:intermittent}
Under Assumption \ref{asmp:Thaler}, there exist disjoint intervals $A_1,\dots,A_d\subset[0,1]$, their complement $Y=[0,1]\setminus(A_1\cup\dots\cup A_d)$, and a Banach space $\cB\subset L^1(\mu)$ such that the following conditions are satisfied.	
\begin{enumerate}
\item $x_j\in A_j\subset J_j$ for $j=1,\dots,d$, and $A_j$ and $J_j\setminus A_j$ have positive Lebesgue measure.
\item If we normalize $\mu$ so that $\mu[Y]=1$, then Assumptions \ref{asmp:dyn-sep}, \ref{asmp:reg-var} and \ref{asmp:ape-ren} are satisfied and hence Theorems \ref{Thm:conv-of-density} and \ref{Thm:extend-observable} can be applied. 

\item Let $v\in \operatorname{Lip}([0,1])$ and $u=v/h$, then $u\in \cB([0,1])$, where $\cB([0,1])$ is the function space given by \eqref{B(X)} with $X=[0,1]$. Moreover, the condition \eqref{Thm:extend-observable-1} is satisfied.

\item Let $v$ be a Riemann integrable function on $[0,1]$ and $u=v/h$. Then, for any bounded continuous function $g:[0,1]^d\times [0,\infty)\to \bR$,
\begin{equation}
\begin{split}\label{Thm:intermittent-1}
&\frac{\Gamma(\alpha)n}
   {b_n} 1_Y L^n
    \bigg[g\bigg(\bigg(\frac{S_{A_j}(n)}{n}\bigg)_{j=1}^d, \frac{S_Y(n)}{b_n}\bigg) u \bigg]
    \\
    &\to
    \bE\Big[g\big((U_j)_{j=1}^d, W\big)\Big]\int_X u \:\rd \mu,	
    \quad
    \text{in $L^\infty(Y)$, as $n\to\infty$.}
\end{split}
\end{equation}
\end{enumerate}
\end{Thm}

Theorem \ref{Thm:intermittent} yields the following corollary immediately.

\begin{Cor}
Let $\nu(\rd x)$ be a probability measure on $[0,1]$ which admits a Riemann integrable density $\nu(\rd x)/\rd x=v$ with respect to the Lebesgue measure $\rd x$ on $[0,1]$. (Hence $\nu(\rd x)/\mu(\rd x)=v/h$.) Then, for any Borel subset $A\subset Y$ with positive Lebesgue measure,
\begin{equation}
	\nu[f^{-n}(A)]\sim \frac{b_n}{\Gamma(\alpha)n}\mu(A)\sim
	\frac{\sin(\pi\alpha)}{\pi}\frac{1}{n^{1-\alpha}\ell(n)}\mu(A),
	 \quad \text{as $n\to\infty$.}
\end{equation}
In addition, for any $t_1,\dots,t_d\in[0,1]$ and $t_{d+1}\geq0$, 
\begin{equation}
\begin{split}
    &\nu\bigg[\frac{S_{A_1}(n)}{n}\leq t_1,\dots,  \frac{S_{A_d}(n)}{n}\leq t_d,\; 
    \frac{S_{Y}(n)}{b_n}\leq t_{d+1}
    \:\bigg|\:f^{-n}(A)\bigg]
    \\[5pt]
    &\to
    \bP[U_1\leq t_1,\dots,U_d\leq t_d,\; W\leq t_{d+1}],	
    \quad \text{as $n\to\infty$.}
\end{split}	
\end{equation}

\end{Cor}

We divide the proof of Theorem \ref{Thm:intermittent} by cases $d=2$ and $d\geq3$.

\subsection{Proof of Theorem \ref{Thm:intermittent} in the case $d=2$}

If  $d=2$ then $x_1=0$ and $x_2=1$. We can take $2$-periodic points $\gamma\in J_1\setminus\{0\}$ and $f(\gamma)\in J_2\setminus\{1\}$, since $f^2:[0,1]\to[0,1]$ has four full branches and $(f^2)'(x)\geq1$.  In this case we take   
\begin{align}
	A_1=[0,\gamma),
	\quad
	Y= [\gamma, f(\gamma)],
	\quad\text{and}\quad
	A_2=(f(\gamma),1],
\end{align}
as in \cite[Section 4]{Tha02}. We normalize $\mu$ so that $\mu[Y]
=1$. 
Then Assumption \ref{asmp:dyn-sep} is satisfied. 
%We define $\varphi$, $S_j$, $r_{j,n}$, $F$, $R$, $R_{j,n},R(z), T_{n,g}$ as in Section \ref{sec:setting}. 
Recall that $L$ denotes the dual operator of $f$ with respect to $\mu$.
As in the proofs of \cite[Lemma 3]{Tha02} or \cite[Theorem 8.1]{ThZw}, there exists constants $c_1'',c_2''>0$ such that, as $n\to\infty$,
\begin{align}
  r_{1,n}
  &=\mu[Y\cap \{S_{A_1}(\varphi)=n\}]
  =
  \int_{Y}L^{n+1} 1_{Y\cap J_2\cap \{\varphi=n+1\}}\rd\mu
  \sim 
  c_1''(f_1^{-n}(1)-f_1^{-n-1}(1)),
  \\
  r_{2,n}
  &=\mu[Y\cap \{S_{A_2}(\varphi)=n\}]
  =
  \int_{Y}L^{n+1} 1_{Y\cap J_1\cap \{\varphi=n+1\}}
  \rd\mu
  \sim 
  c_2''(f_2^{-n-1}(0)-f_2^{-n}(0)).
\end{align}
By substituting $p=1/\alpha$ and $v(x)=f_1^{-1}(x)$ or $v(x)=1-f_2^{-1}(1-x)$ into \cite[Remark 1]{Zwe03}, we see that 
there exists a function $\ell(t)$ slowly varying at $\infty$ and there exist constants $\beta_1,\beta_2\in(0,1)$ with $\beta_1+\beta_2=1$ such that, as $n\to\infty$,
\begin{equation}
   r_{j,n}\sim \alpha \beta_j  n^{-1-\alpha}\ell(n),
    \quad
   \text{for $j=1,2$},
   \end{equation}  
 and hence
 \begin{equation}
 \sum_{k\geq n}r_{j,k} \sim \beta_j n^{-\alpha}\ell(n),
   \quad
   \text{for $j=1,2$}.
\end{equation}
Hence Assumption \ref{asmp:reg-var} is satisfied. (For example, in the case of Boole's transformation (Example \ref{Ex:Gen-Arc}), $\alpha=\beta_1=\beta_2=1/2$ and $\ell(t)$ can be taken as a constant function.)

Let us denote by $F$ the first return map to $Y$, which is uniformly expanding.
Define an interval $P_{j,n}\subset Y$ by
\begin{align}
    	P_{j,n}=Y\cap \Big(\bigcap_{k=1}^n f^{-n}(A_j)\Big) \cap f^{-(n+1)}(Y)=Y\cap \{S_{A_j}(\varphi)=n\},
    	\quad \text{$j=1,2$ and $n\geq1$.}
\end{align}
Then $\bigcup_{j,n} P_{j,n}=Y$, a.e., and $F(P_{j,n})=Y$, a.e.  As stated in \cite[Section 7]{Ser20}, the system $(Y, \mu|_Y, F, (P_{j,n}))$ is a mixing, probability preserving, Gibbs--Markov map in the sense of \cite{AaDe}. Set $Y'= \bigcap_{n\geq0}F^{-n}(\bigcup_{j,n} P_{j,n})$. Then $Y'=Y$, a.e.
For $x,y\in Y'$, we denote by $t(x,y)$ the separation time:
\begin{align}
t(x,y)=\sup\{m\geq 0: x,y\in \bigcap_{k=1}^{m}F^{-k+1}(P_{j_k, n_k}) \;\text{for some $j_k=1,2$ and $n_k\geq1$}\}.
\end{align}
Take $\theta\in(0,1)$ so that $\theta\leq 1/(\inf_{x\in Y'}F'(x))\in(0,1)$. For a measurable function $u:Y'\to\bC$, we define $\|u\|_{\cB}$ by
\begin{align}
    \|u\|_{\cB}=\sup|u| + \sup_{x,y\in Y'} \frac{|u(x)-u(y)|}{\theta^{t(x,y)}}.	
\end{align}
We denote by $\cB$ the space of bounded measurable functions $u:Y'\to \bC$ with $\|u\|_{\cB}<\infty$.
Note that $\cB$ contains any function $u:Y'\to \bC$ which can be extended to a Lipschitz continuous function on $Y$ with respect to the Euclidean norm. 
Indeed, the diameter of $\bigcap_{k=1}^{m}F^{-k+1}(P_{j_k, n_k})$ is $O(\theta^m)$ uniformly in the choice of $j_k, n_k$, and hence \begin{equation}
   \sup_{x,y\in Y'}\frac{|u(x)-u(y)|}{\theta^{t(x,y)}}
   \leq
   \bigg(\sup_{x,y\in Y'}\frac{|u(x)-u(y)|}{|x-y|}\bigg)	\bigg(\sup_{x,y\in Y'}\frac{|x-y|}{\theta^{t(x,y)}}\bigg)<\infty.
\end{equation}

Then we can show that the conditions \eqref{asmp:ape-ren-bdd} and \eqref{asmp:ape-ren-spec} in Assumption \ref{asmp:ape-ren} are satisfied in the same way as in \cite[Section 5]{Sar02}. For the proof of the condition \eqref{asmp:ape-ren-ape} in Assumption \ref{asmp:ape-ren}, we need to show the following lemma. Let $\bS^1=\{z\in \bC: |z|=1\}$.

\begin{Lem}[$F$-aperiodicity]
\label{Lem:TY-ap}
   Assume there exist $(t_1,t_2)\in(-\pi,\pi]^2$, $\lambda_0\in \bS^1$, a measurable function $g:Y'\to \bS^1$ such that
 \begin{equation}
    \exp
    \bigg(
    i\sum_{j=1,2}t_j S_{A_j}(\varphi)
    \bigg)
    =
    \frac{\lambda_0 g}{g\circ F},
    \quad
    \text{on $Y'$.}	
 \end{equation}
Then $t_1=t_2=0$.	
\end{Lem}

\begin{proof}
We imitate the argument of \cite[p.651]{Sar02}. 
By \cite[Corollary 3.2]{AaDe}, the function $g$ must be constant on $Y'$, i.e.,
\begin{equation}
	\exp
    \bigg(
    i\sum_{j=1,2}t_j S_{A_j}(\varphi)
    \bigg)
    =
    \lambda_0.
\end{equation}
Therefore, for $j=1,2$,
\begin{equation}
	\exp
    \bigg(
    i t_j \bigg((S_{A_j}(\varphi))(x)-(S_{A_j}(\varphi))(y)\bigg)
    \bigg)
    =
    1,
    \quad
    x,y\in \bigcup_{n\geq1}P_{j,n}\cap Y'.
\end{equation}
Since
\begin{equation}
   \bigg\{(S_{A_j}(\varphi))(x)-(S_{A_j}(\varphi))(y) :
   x,y\in \bigcup_{n\geq1}P_{j,n}\cap Y'\bigg\}
   =\bZ,	
\end{equation}
we conclude $t_j=0$.	
\end{proof}

By Lemma \ref{Lem:TY-ap} and \cite[Proposition 3.7]{AaDe}, we see that the condition \eqref{asmp:ape-ren-ape} in Assumption \ref{asmp:ape-ren} is also satisfied. See also \cite[p.651]{Sar02}.
Therefore we can apply Theorems \ref{Thm:conv-of-density} and \ref{Thm:extend-observable} to this setting. 
%Moreover, we can prove the convergence \eqref{Cor:conditional-1} for $\nu(\rd x)=u(x)\mu(\rd x)$ where $u:[0,1]\to [0,\infty)$ denotes a Riemann-integrable function and for Borel subset $A\subset(0,1)$ with $\bar{A}\subset(0,1)$. Its proof is  the same as that of \cite[Theorem 11.8]{MelTer}.  

Let $v\in \operatorname{Lip}([0,1])$ and $u=v/h$. As stated in the proof of  \cite[Theorem 11.8]{MelTer}, $u\in\cB([0,1])$ and
\begin{equation}
	\|L^m u_{j,m}\|_{\cB}=O(r_{j,m})=O(m^{-1-\alpha}\ell(m)),
	\quad\text{as $m\to\infty$ ($j=1,\dots,d$),}
\end{equation}
and hence the condition \eqref{Thm:extend-observable-1} is satisfied.

Let $v$ be  Riemann integrable on $[0,1]$ and $u=v/h$. For any $\varepsilon>0$, we can take $v^{\pm}\in \operatorname{Lip}([0,1])$ so that
\begin{equation}
    v^-\leq v \leq v^+
    \quad\text{and}\quad
    \int_0^1(v^+(x) - v^-(x))\rd x\leq\varepsilon.	
\end{equation}
By applying Theorem \ref{Thm:extend-observable} to $v^{\pm}/h$ and the positivity of $L$, we see that
\begin{equation}
\begin{split}
	&\limsup_{n\to\infty}\|\text{(L.H.S.\ of \eqref{Thm:intermittent-1})}-\text{(R.H.S.\ of \eqref{Thm:intermittent-1})} \|_{L^\infty(Y)}
	\\
	&\leq \sup|g|\int_0^1(v^+(x) - v^-(x))\rd x
	\leq \sup|g|\varepsilon.	
\end{split}
\end{equation}
Since $\varepsilon>0$ was arbitrary, we obtain \eqref{Thm:intermittent-1}. We now complete the proof of Theorem \ref{Thm:intermittent} in the case $d=2$.

\subsection{Proof of Theorem \ref{Thm:intermittent} in the case $d\geq3$}

In the case $d\geq3$, let us take
\begin{align}\label{choice-of-partition}
	A_j= J_j\cap f^{-1}(J_j), \;\; \text{for $j=1,\dots,d$},\quad
	\text{and}\quad
	Y=[0,1]\setminus \bigcup_{j=1}^d A_j. 
\end{align}
We normalize $\mu$ so that $\mu[Y]=1$. Similarly as in the case $d=2$, we see that Assumptions \ref{asmp:dyn-sep} and \ref{asmp:reg-var} are satisfied.

Denote by $F$ the first return map to $Y$, which is uniformly expanding. Let us construct a Markov partition of $F$.
Set
\begin{equation}
Y_{j,-}= Y\cap J_j\cap(0,x_j), \quad Y_{j,+}=Y\cap J_j\cap (x_j,1).
\end{equation}
 For $j,k=1,\dots,d$ and $\sigma=\pm$ and $n\geq0$, set
\begin{equation}
P_{j,k,\sigma, n}= Y\cap J_j\cap \{S_{A_k}(\varphi)=n\} \cap F^{-1}(Y_{j,\sigma}).
\end{equation}
We see $Y_{j,k,\sigma,n}\neq \emptyset$ if and only if $j\neq k$ and $(k,\sigma)\neq (1,-), (d,+)$. 
Then
\begin{equation}
	F(P_{j, k,\sigma,n})=Y_{k,\sigma}, \quad
	\text{a.e.,
	\;if $P_{j,k,\sigma,n}\neq\emptyset$,}
\end{equation} 
and 
\begin{equation}
  Y_{j,-}=\bigcup_{\substack{k,\sigma,n\\k<j}}P_{j,k,\sigma,n},
  \quad
  Y_{j,+}=\bigcup_{\substack{k,\sigma,n\\k>j}} P_{j,k,\sigma,n}.	
\end{equation}
The system $(Y, \cB(Y), \mu|_Y, F, (Y_{i,j,\sigma,n}))$ is a mixing, probability preserving, Gibbs--Markov map as stated in \cite[Section 7]{Ser20}.
Similarly as in the case $d=2$, we need to show the following lemma.

\begin{Lem}[$F$-aperiodicity]
\label{Lem:F-ap-3}
   Assume there exist $(t_1,\dots,t_d)\in(-\pi,\pi]^d$, $\lambda_0\in \bS^1$, a measurable function $g:Y'\to \bS^1$ such that
 \begin{equation}
    \exp
    \bigg(
    i\sum_{j=1}^d
    t_j S_{A_j}(\varphi)
    \bigg)
    =
    \frac{\lambda_0 g}{g\circ F},
    \quad
    \text{on $Y'$.}	
 \end{equation}
Then $t_1=\dots=t_d=0$.	
\end{Lem}

\begin{proof}
	By \cite[Theorem 3.1]{AaDe}, the function $g$ must be constant on each $Y_{j,\sigma}\cap Y'$, which implies $g/(g\circ F)$ must be constant on each $\bigcup_{\sigma, n}P_{j,k,\sigma,n}\cap Y'$ for $j\neq k$. 
Hence, for $j\neq k$,
\begin{equation}
  \exp
    \bigg(
    i
    t_k \bigg((S_{A_k}(\varphi))(x)
    -
    (S_{A_k}(\varphi))(y)
    \bigg)
    \bigg)
    =1,
    \quad
    x,y
    \in 
    \bigcup_{\sigma, n}P_{j,k,\sigma,n}\cap Y'.	
\end{equation}
Since
\begin{equation}
   \bigg\{(S_{A_j}(\varphi))(x)-(S_{A_j}(\varphi))(y) :
   x,y\in \bigcup_{\sigma, n}P_{j,k,\sigma,n}\cap Y'\bigg\}
   =\bZ,	
\end{equation}
we conclude $t_k=0$ ($k=1,\dots,d$).	
\end{proof}

 In a similar way as in the case $d=2$, we can construct a suitable Banach space $(\cB, \|\cdot\|_{\cB})$ so that Assumption \ref{asmp:ape-ren} will be satisfied. 
Therefore we can apply Theorems \ref{Thm:conv-of-density} and \ref{Thm:extend-observable} to this setting. The other part of the proof is the same as  in the case $d=2$, so we omit it.

\subsection*{Acknowledgements}

The research of Toru Sera was partially supported by JSPS KAKENHI Grant Numbers JP19J11798 and JP21J00015, and by the Research Institute for Mathematical Sciences, an International Joint Usage/Research Center located in Kyoto University.

%%%%% references %%%%%

%\bibliographystyle{alpha}

\end{document}